\documentclass[12pt]{article}
\usepackage{mathtools}
\usepackage{amsmath,diagbox}
\usepackage{indentfirst}
\usepackage{mathrsfs}
\usepackage{amsfonts}
\usepackage{arydshln}
\usepackage{enumerate}
\usepackage{setspace}
\usepackage{amssymb,amsthm,cases}
\usepackage{amsbsy,pifont}
\usepackage{latexsym,diagbox,tabularx}
\usepackage{amsmath,amscd,bbm,multirow}
\usepackage{graphicx,epsfig,extarrows,dsfont,mathtools}
\usepackage[final,allcolors=blue,colorlinks=true]{hyperref}
\usepackage[normalem]{ulem}
\input xy
\usepackage{caption,subcaption} 
\usepackage{tikz}
\usepackage{tikz-cd}
\usetikzlibrary{decorations.pathreplacing,decorations.markings}
\usetikzlibrary{cd}
\usetikzlibrary{shapes.geometric}
\usetikzlibrary{arrows,arrows.meta}
\usetikzlibrary{graphs,positioning}
\usetikzlibrary{matrix,decorations.pathmorphing}
\usetikzlibrary{math,calc,intersections,through,angles,arrows.meta,shapes.geometric,shadows,quotes,spy,datavisualization,datavisualization.formats.functions,plotmarks}
\tikzset{every picture/.style={samples=300,smooth,line join=round,thick,>=stealth}}
\usepackage{enumerate}
\usepackage{fancyhdr}
\usepackage[sectionbib]{chapterbib}

\tikzset{
	on each segment/.style={
		decorate,
		decoration={
			show path construction,
			moveto code={},
			lineto code={
				\path [#1]
				(\tikzinputsegmentfirst) -- (\tikzinputsegmentlast);
			},
			curveto code={
				\path [#1] (\tikzinputsegmentfirst)
				.. controls
				(\tikzinputsegmentsupporta) and (\tikzinputsegmentsupportb)
				..
				(\tikzinputsegmentlast);
			},
			closepath code={
				\path [#1]
				(\tikzinputsegmentfirst) -- (\tikzinputsegmentlast);
			},
		},
	},
	mid arrow/.style={postaction={decorate,decoration={
				markings,
				mark=at position .5 with {\arrow[#1]{stealth}}
	}}},
}

\numberwithin{equation}{section}
\theoremstyle{plain}
\newtheorem{theorem}{Theorem}[section]

\newtheorem{corollary}[theorem]{Corollary}

\newtheorem{lemma}[theorem]{Lemma}
\newtheorem{proposition}[theorem]{Proposition}

\theoremstyle{definition}
\newtheorem{definition}[theorem]{Definition}

\newtheorem{remark}[theorem]{Remark}

\newtheorem{notation}[theorem]{Notation}

\def\XXint#1#2#3{{\setbox0=\hbox{$#1{#2#3}{\int}$}
		\vcenter{\hbox{$#2#3$}}\kern-.5\wd0}}

\DeclareMathSymbol{\subseteq}{\mathrel}{symbols}{"12}
\DeclareMathSymbol{\supseteq}{\mathrel}{symbols}{"13} 
\DeclareMathSymbol{\subsetneq}{\mathrel}{AMSb}{"28}                  \DeclareMathSymbol{\supsetneq}{\mathrel}{AMSb}{"29}    
        
\DeclareMathSymbol{\nsubseteq}{\mathrel}{AMSb}{"2A}                  

\DeclareMathSymbol{\nsupseteq}{\mathrel}{AMSb}{"2B}

\DeclareMathDelimiter{\langle}{\mathop}{symbols}{"68}{largesymbols}{"0A}
\DeclareMathDelimiter{\rangle}{\mathclose}{symbols}{"69}{largesymbols}{"0B}
\DeclareSymbolFont{txfontsA}{U}{txmia}{m}{it}
\SetSymbolFont{txfontsA}{bold}{U}{txmia}{bx}{it}
\DeclareFontSubstitution{U}{txmia}{m}{it}
\DeclareMathSymbol{\upalpha}{\mathord}{txfontsA}{"0B}
\DeclareMathSymbol{\upbeta}{\mathord}{txfontsA}{"0C}
\DeclareMathSymbol{\upgamma}{\mathord}{txfontsA}{"0D}
\DeclareMathSymbol{\updelta}{\mathord}{txfontsA}{"0E}
\DeclareMathSymbol{\upepsilon}{\mathord}{txfontsA}{"0F}
\DeclareMathSymbol{\upzeta}{\mathord}{txfontsA}{"10}
\DeclareMathSymbol{\upeta}{\mathord}{txfontsA}{"11}
\DeclareMathSymbol{\uptheta}{\mathord}{txfontsA}{"12}
\DeclareMathSymbol{\upiota}{\mathord}{txfontsA}{"13}
\DeclareMathSymbol{\upkappa}{\mathord}{txfontsA}{"14}
\DeclareMathSymbol{\uplambda}{\mathord}{txfontsA}{"15}
\DeclareMathSymbol{\upmu}{\mathord}{txfontsA}{"16}
\DeclareMathSymbol{\upnu}{\mathord}{txfontsA}{"17}
\DeclareMathSymbol{\upxi}{\mathord}{txfontsA}{"18}
\DeclareMathSymbol{\uppi}{\mathord}{txfontsA}{"19}
\DeclareMathSymbol{\uprho}{\mathord}{txfontsA}{"1A}
\DeclareMathSymbol{\upsigma}{\mathord}{txfontsA}{"1B}
\DeclareMathSymbol{\uptau}{\mathord}{txfontsA}{"1C}
\DeclareMathSymbol{\upupsilon}{\mathord}{txfontsA}{"1D}
\DeclareMathSymbol{\upphi}{\mathord}{txfontsA}{"1E}
\DeclareMathSymbol{\upchi}{\mathord}{txfontsA}{"1F}
\DeclareMathSymbol{\uppsi}{\mathord}{txfontsA}{"20}
\DeclareMathSymbol{\upomega}{\mathord}{txfontsA}{"21}
\DeclareMathSymbol{\upvarepsilon}{\mathord}{txfontsA}{"22}
\DeclareMathSymbol{\upvartheta}{\mathord}{txfontsA}{"23}
\DeclareMathSymbol{\upvarpi}{\mathord}{txfontsA}{"24}
\DeclareMathSymbol{\upvarrho}{\mathord}{txfontsA}{"25}
\DeclareMathSymbol{\upvarsigma}{\mathord}{txfontsA}{"26}
\DeclareMathSymbol{\upvarphi}{\mathord}{txfontsA}{"27}                   

\DeclareSymbolFont{ugmL}{OMX}{mdugm}{m}{n}
\SetSymbolFont{ugmL}{bold}{OMX}{mdugm}{b}{n}
\DeclareMathAccent{\wideparen}{\mathord}{ugmL}{"F3}

\def\pt{\partial}
\def\mc{\mathcal}
\def\ra{\rightarrow}

\def\s{\subseteq}

\def\e{\epsilon}

\def\ol{\overline}

\def\vp{\varphi}

\def\bf{\textbf}
\def\pt{\partial}

\def\Om{\Omega}
\def\la{\lambda}
\def\al{\alpha}
\def\be{\beta}
\def\de{\delta}

\def\De{\Delta}

\def\Ga{\Gamma}

\def\ts{\times}

\def\iy{\infty}

\def\f{\frac}

\def\Lra{\Leftrightarrow}

\def\df{\mathrm d}

\def\wt{\widetilde}
\def\wh{\widehat}

\def\esssup{\operatorname*{ess\ \! sup}}

\def\hra{\hookrightarrow}

\def\mcA{\mathcal{A}}

\def\mcM{\mathcal{M}}

	\DeclareMathOperator{\im}{i}

	\DeclareMathOperator{\Div}{div}

	\DeclareMathOperator{\Real}{Re}
	\DeclareMathOperator{\Image}{Im}

	\DeclareMathOperator{\supp}{{supp}}

	\newcommand{\R}{\mathbb R}
	\newcommand{\N}{\mathbb N}

	\newcommand{\Cc}{\mathbb C}

\begin{document}
	
\title{{   Observability Inequality for a Multi-Dimensional Degenerate
Schr\"odinger Equation with an Interior Degeneracy Point}}
\date{}

\author{ Dong-Hui Yang$^{a}$, Ghadir Shokor$^{b}$\footnote{\small
		The corresponding author. Email: ghadirshokor1@gmail.com}, Mengze Gu$^c$, Jie Zhong$^{d}$\\
	$^a${\it School of Mathematics and Statistics, Central South University}\\
	{\it Changsha 410075, P.R.China}\\
	$^a${\it School of Mathematics and Statistics, Beijing Institute of Technology}\\
	{\it Beijing 100000, China}\\
	$^c${\it School of Mathematics and Science, Changsha Normal University}\\ {\it  Changsha 410083, China}\\
	$^d${\it California State University Los Angeles}\\
	{\it Los Angeles, 90032, USA}\\  }
	
	\maketitle{}
	\thispagestyle{empty}

\begin{abstract}
We study the observability inequality for a multi-dimensional degenerate Schr\"odinger equation possessing an interior degeneracy point, where the principal coefficient vanishes at the origin inside the domain and breaks ellipticity. Within a weighted Sobolev space, we establish the spectral properties, well-posedness, and regularity estimates for the degenerate operator. To tackle the failure of classical controllability tools, we develop a regularization approximation scheme. We construct a sequence of non-degenerate approximating problems. We then establish uniform a-priori estimates and solution convergence. Furthermore, we derive an observability inequality for the regularized problems and pass to the limit. This yields a non-standard observability inequality for the original degenerate equation for sufficiently large observation time, which quantifies the influence of the degeneracy strength on the observation time. Crucially, we establish the precise parameter restriction $3\alpha + N > 4$ to ensure that every test function in $C_0^\infty(\Omega)$ belongs to $D(\mathcal A^{\frac 32})$.

\vspace{0.3cm}

	
\vspace{0.3cm}
		

\end{abstract}
	
\section{Introduction}\label{S1}
The controllability of partial differential equations is a central theme in mathematical control theory. It connects to inverse problems, stabilization, and quantum dynamical systems \cite{BVCS, BIUBMS,  CPADMY, Lions, ZEDET}.  A fundamental tool in the study of controllability for deterministic systems is the Hilbert Uniqueness Method (HUM), introduced by J.-L. Lions, which establishes a connection between exact controllability of a system and an observability inequality for its corresponding adjoint system. Consequently, the study of controllability can often be reduced to establishing suitable observability estimates, \cite{Lions,Yang2}. In the present paper, we only focus on establishing the observability inequality for our problem.

In the physical sciences, the observability of a system governed by a Schr\"odinger equation concerns the ability to reconstruct its full state from localized boundary measurements. However, observability heavily depends on the spatial behavior of the equation's coefficients. The most challenging case arises when the principal symbol degenerates at an interior point. While boundary degeneracies can often be localized and handled near the boundary, interior degeneracies admit no such spatial reduction. In this situation, classical observability tools fail because the vanishing principal symbol destroys ellipticity \cite{CPRFPMN} beyond the repair of standard perturbation or cut-off arguments \cite{Yang}.

To meet this challenge, we consider the following problem,
\begin{equation}\label{12.14.1}
	\begin{cases}
		\im \partial_t\varphi-\Div(|x|^\alpha \nabla \varphi)=f, &\text{in }Q, \\
		\varphi=0, &\text{on }\partial Q, \\
		\varphi(0)=\varphi^0, &\text{in }\Omega,
	\end{cases}
\end{equation}
where $\im$ is a 
pure imaginary number with $\im^2=-1$, $\alpha\in (0,2)$ is a given constant, $\Omega\subset\mathbb{R}^N$ is a bounded domain containing $0$ with $\partial\Omega\in C^3$, $Q=\Omega\times (0,T)$ with $T>0$ being a constant, $\partial Q = \partial \Omega \ts (0,T)$, $\varphi^0\in H_0^1(\Om;w)$ is the initial data, and $f\in L^2(0,T;H_0^1(\Omega;w))$. Here,  the weight Sobolev space $H_0^1(\Om;w)$ will be defined in \eqref{04.07.2}.

Assume  there exists $R>0$ such that $B_{4R}(0)\subset \Omega$, where $B_R(x_0)=\{x\in\R^N\colon |x-x_0|<R\}$ and we also denote $B_R=B_R(0)$.
Denote
\begin{equation}\label{gbz1}
	w=|x|^\alpha,\mbox{ and } \mcA \vp=-\Div(w\nabla \vp).
\end{equation}

The primary objective of this paper is to derive an observability inequality, which bounds the initial energy of the system by boundary measurements of the solution. While observability theory for non-degenerate Schrödinger and hyperbolic equations is well-established via Carleman estimates and multiplier techniques \cite{LIARTR, Machtyngier}, results for degenerate equations remain limited. In one dimension, boundary degeneracy for singular Schrödinger, wave, and parabolic systems has been widely investigated \cite{AFPCGL, CPFG, CPMPV, Fragnelli, FGDMAS, GMEB, LTLRT}, and interior degeneracy has been studied in parabolic settings \cite{CPRFPMN}. Nevertheless, the multi-dimensional Schrödinger equation with an interior degeneracy point remains completely open \cite{YDBGJZL}. Extending existing tools to higher dimensions presents severe mathematical barriers. For instance, the Carleman estimates break down near the interior singularity \cite{RLECA, Wu1}, and the classical Holmgren theorem is not directly applicable; consequently, classical unique continuation arguments needed for standard observability inequalities cannot be used. Additionally, the analysis must be conducted in weighted Sobolev spaces where smooth, compactly supported functions may fail to belong to the domains of fractional powers of the degenerate operator.

Our main result is a non-standard observability inequality that relates the initial energy of the system to boundary measurements of the solution. This inequality explicitly quantifies the influence of the degeneracy strength on the observation time. More precisely, the observability inequality holds for sufficiently large observation time, under the parameter condition $3\alpha+N>4$. These restrictions reflect the analytical difficulties caused by the interior degeneracy. These results provide an initial step toward understanding the boundary observability of multi-dimensional degenerate Schr\"odinger equations.

The paper is organized as follows. Section \ref{S2} introduces the weighted Sobolev spaces, establishes the spectral properties of the degenerate operator. It also proves well-posedness of the equation, and derives the required regularity estimates. Section \ref{S3} presents the approximation scheme. We define regularized operators, prove uniform estimates, and establish convergence of the approximating solutions. In Section \ref{S4}, we derive the observability inequality for the approximating problems. We pass to the limit, and prove the observability of the degenerate equation. Finally, Section \ref{S5} concludes the paper with future work.

The main results of this paper are Theorems \ref{04.13.T1}, \ref{04.14.T1}, and \ref{04.15.T2}. The strong convergence assertion \eqref{04.14.3} in Theorem \ref{04.14.T1} and the observability inequality in Theorem \ref{04.15.T2} are established under the parameter condition $3\alpha+N>4$.

\section{Solution spaces}\label{S2}
\subsection{Solution spaces}
In this paper, all the spaces are complex spaces. e.g., 
\begin{equation*}
	L^2(\Om)=L^2(\Om,\Cc)=\left\{u:\Om\ra\Cc\colon \int_\Om | u |^2\df x<+\iy\right\}
\end{equation*}
and its inner product is 
\begin{equation*}
	( u,v)_{L^2(\Om)}= \int_\Om u\ol v\df x, 
\end{equation*}
where $\Cc=\R+\im\R$ is the complex numbers space, $\ol u$ is the conjugate function of $u$, and $\Real u$ is the real part of the function $u$. We denote by $\Image u$ the imaginary part of the function $u$.

We define
\begin{equation*}
H^{1}(\Omega; w) = \left\{u \in L^2(\Omega) \colon  \int_\Om w\nabla u\cdot \ol{\nabla  u}\df x<+\iy\right\},
\end{equation*}
where $\nabla u=(\f{\pt u}{\pt x_1}, \cdots, \f{\pt u}{\pt x_N})$ are the distributional derivatives with respect to the spatial variables $x_1, \cdots, x_N$. The inner product on $H^1(\Omega; w)$ is
\begin{equation*}
( u, v)_{H^1(\Omega; w)} = \int_\Omega u \ol v \, \mathrm{d}x +   \int_\Omega w\nabla u\cdot\ol{ \nabla v} \, \mathrm{d}x,
\end{equation*}
and the norm is
\begin{equation*}
\|u\|_{H^1(\Omega; w)} = \left(\int_\Omega |u|^2  \, \mathrm{d}x +   \int_\Omega w\nabla u\cdot \ol{\nabla u} \, \mathrm{d}x\right)^{\frac{1}{2}}.
\end{equation*}
We define
\begin{equation}\label{04.07.2}
H_{0}^1(\Omega; w) = \text{the closure of } \mathcal{D}(\Omega,\Cc) \text{ in } H^1(\Omega; w),
\end{equation}
where $\mc{D}(\Om)=\mathcal{D}(\Omega,\Cc) = C_0^\infty(\Omega,\Cc)$ is the space of test functions.
We denote by $H^{-1}(\Omega; w)$ the dual space of $H_{0}^1(\Omega; w)$, which is a subspace of $\mathcal{D}'(\Omega)$.
 
We define
\begin{equation*}
	H^2(\Om;w)=\left\{u\in H^1(\Om;w)\colon \mcA u\in L^2(\Om)\right\}.
\end{equation*}
The inner product on $H^2(\Om;w)$ is
\begin{equation*}
(u,v)_{H^2(\Om;w)}= \int_\Om (\mcA u)\ol{(\mcA v)}\df x+ ( u,v)_{H^1(\Om;w)},
\end{equation*}
and the norm is
\begin{equation*}
	\|u\|_{H^2(\Om;w)}=\left(\int_\Om |u|^2\df x+ \int_\Om w |\nabla u|^2\df x+\int_\Om |\mcA u|^2\df x\right)^\f{1}{2}.
\end{equation*}

It is also well-known that
\begin{equation*}
	L^2(\Om), \quad H^1(\Om;w),\quad H^2(\Om;w)
\end{equation*} 
are Hilbert spaces. (See \cite{GC,Heinonen}.)

Finally, we define
\begin{equation*}
	D(\mcA)=H_0^1(\Om;w)\cap H^2(\Om;w).
\end{equation*}

 \begin{lemma}\label{08.15.L1}
Let $N \geq 2$ and $\alpha \in (0, 2)$. Then, for all $u \in H_0^1(\Omega; w)$, the following inequality holds:
\begin{equation*}
(N - 2 + \alpha)^2\int_\Om |x|^{\al-2} |u|^2\df x \leq 4\int_\Om w\nabla u\cdot \nabla \ol u\df x.
\end{equation*}
Furthermore, if $u \in H_0^1(\Omega; w)$, it follows that $u \in L^2(\Omega)$.
\end{lemma}

\begin{proof}
We prove the first part. 
For the real case, the first part of the lemma can be derived from \cite[Lemma 3.1]{Wu} or \cite[Proposition 2.1 (1)]{Stuart}. For the complex case, note that for each $u\in H_0^1(\Om;w)$ we have $u=u_1+\im u_2$ ($u_1, u_2$ are real functions), we get
\begin{equation*}
	\begin{split}
		(N-2+\al)^2\int_\Om |x|^{\al-2}|u|^2\df x
		&=(N-2+\al)^2\int_\Om |x|^{\al-2}\left(u_1^2+u_2^2\right)\df x\\
		&\leq 4\int_\Om w\nabla u_1\cdot \nabla u_1\df x+4\int_\Om w\nabla u_2\cdot \nabla u_2\df x\\
		&=4\int_\Om w\nabla u\cdot \nabla \ol u\df x. 
	\end{split}
\end{equation*}

For the proof of the second part, from Lemma \ref{08.15.L1} and $\alpha \in (0, 2)$, we obtain 
\begin{equation}\label{08.15.10}
\frac{N - 2 + \alpha}{2M^{1 - \frac{\alpha}{2}}} \|u\|_{L^2(\Omega)} \leq \left(\int_\Om w\nabla u\cdot \nabla\ol u\df x\right)^\f{1}{2}, \mbox{ with } M\equiv \sup_{x\in\Om}|x|+1,
\end{equation}
which is the Poincar\'{e} inequality.
\end{proof}

\begin{remark}\label{08.16.R1}
From Lemma \ref{08.15.L1},  the norm
\begin{equation}\label{08.19.2}
\|u\|_{H_0^1(\Omega; w)} = \left(\int_\Omega w\nabla u \cdot \nabla \ol u \, \mathrm{d}x\right)^{\frac{1}{2}}
\end{equation}
is an equivalent norm in $H_0^1(\Omega; w)$. Hereafter, we use \eqref{08.19.2} to define the norm of $H_0^1(\Omega; w)$.
\end{remark}

The following Lemma \ref{08.15.L4} is \cite[Theorem 3.4]{Wu} or  \cite[Proposition 2.1 (5)]{Stuart}.
\begin{lemma}\label{08.15.L4}
The embedding $H_0^1(\Omega; w) \hookrightarrow L^2(\Omega)$ is compact.
\end{lemma}

\begin{proof}
	This is Theorem 3.4 in \cite{Wu}. 
\end{proof}

\subsection{Spectrum}
\begin{notation}\label{04.07.N1}
	Define the sesquilinear form
	\begin{equation}\label{09.15.5}
		a(u,v)
		=
		\int_\Omega \nabla u\cdot w \nabla\overline{v}\,dx,
		\quad
		u,v\in H_0^1(\Omega;w).
	\end{equation}
	The form $a$ is densely defined, Hermitian, positive, closed, and coercive in $L^2(\Omega)$, with form domain
	\[
	D(a)=H_0^1(\Omega;w).
	\]
	Hence, by the representation theorem for closed sesquilinear forms, there exists a unique positive self-adjoint operator $\mathcal A$ in $L^2(\Omega)$ associated with $a$, characterized by
	\[
	a(u,v)=(\mathcal A u,v)_{L^2(\Omega)},
	\qquad
	u\in D(\mathcal A),\quad v\in H_0^1(\Omega;w).
	\]
By the density of $C_0^\infty(\Omega)$ in $H_0^1(\Omega;w)$ and the definition of distributional divergence, the form-associated operator coincides with the previously defined differential operator $\mathcal A$, with the same domain. Moreover, since the embedding
	\[
	H_0^1(\Omega;w)\hookrightarrow L^2(\Omega)
	\]
is compact by Lemma \ref{08.15.L4}, the resolvent of $\mathcal A$ is compact. Consequently, $\mathcal A$ has a discrete spectrum
	\begin{equation*}
		0<\la_1\leq \la_2\leq \cdots \ra +\iy.
	\end{equation*}
	In particular, $\la_n\ (n\in\N^*)$ and the corresponding eigenfunctions $\Phi_n$ satisfy
	\begin{equation}\label{04.07.1}
		\begin{cases}
			\mcA \Phi_n=\la_n \Phi_n, &\mbox{in }\Om,\\
			\Phi_n=0, &\mbox{on }\pt\Om.
		\end{cases}
	\end{equation}
	The eigenfunctions $\{\Phi_n\}_{n\in\N^*}$ can be chosen to form a real-valued orthonormal basis of $L^2(\Omega)$ and an orthogonal basis of $H_0^1(\Omega;w)$. It should be noted that
	\begin{equation}\label{04.10.7}
		\la_1=\inf_{0\neq u\in H_0^1(\Om;w)}
		\f{\int_\Om w\nabla u\cdot \ol{\nabla u}\df x}
		{\int_\Om |u|^2\df x}
		\geq
		\left(\f{N-2+\al}{2M^{1-\f{\al}{2}}}\right)^2
	\end{equation}
	by \eqref{08.15.10}.
\end{notation}

\begin{definition}\label{04.07.D2}
Let $s\geq 0$. We denote 
\begin{equation*}
	D(\mcA^s)=\left\{u\in L^2(\Om)\colon \sum_{n=1}^\iy \la_n^{2s}|u_n|^2< \iy \mbox{ for } u_n = (u,\Phi_n)_{L^2(\Omega)}, n \in \mathbb N^* \right\}. 
\end{equation*}
i.e., $u\in D(\mcA^s)$ if and only if $(\la_n^su_n)_{n\in\N^*}\in \ell^2(\Cc)$.  Its inner product and norm are defined by 
\begin{equation*}
(u,v)_{D(\mathcal A^s)}= 	\sum_{n=1}^{\infty} \lambda_n^{2s} u_n \bar{v}_n, \quad \|u\|_{D(\mathcal A^s)} = \left(\sum_{n=1}^\iy \la_n^{2s}|u_n|^2\right)^{\frac 12},
\end{equation*}
where $u_n = (u, \Phi_n)_{L^2(\Omega)}$, $v_n = (v, \Phi_n)_{L^2(\Omega)}$.

For $s\in [-1,0]$, $D(\mathcal A^s)$ denotes the dual space of
$D(\mathcal A^{-s})$, that is,
\begin{equation*}
	D(\mathcal A^s):=\left(D(\mathcal A^{-s})\right)'.
\end{equation*}
For $u\in D(\mathcal A^s)$, let
\begin{equation*}
u_n	=	\langle u,\Phi_n\rangle_{D(\mathcal A^s),D(\mathcal A^{-s})},	\qquad n\in\mathbb N^*,
\end{equation*}
where
$\langle\cdot,\cdot\rangle_{D(\mathcal A^s),D(\mathcal A^{-s})}$
denotes the duality pairing between $D(\mathcal A^s)$ and
$D(\mathcal A^{-s})$.
Then the norm of $u$ in $D(\mathcal A^s)$ is defined by
\begin{equation*}
	\|u\|_{D(\mathcal A^s)}
	=
	\left(
	\sum_{n=1}^{\iy}
	\lambda_n^{2s}|u_n|^2
	\right)^{\frac12}.
\end{equation*}
\end{definition}

\begin{lemma}\label{06.28.L1}
	Let $u=\sum_{i=1}^\iy u_i \Phi_i\in H_0^1(\Om;w)$ with $u_i=(u,\Phi_i)_{L^2(\Om)}$ for all $i\in\N^*$. We have 
	
	(i) $\nabla u=\sum_{i=1}^\iy u_i\nabla \Phi_i$ and $\|u\|_{H_0^1(\Om;w)}=(\sum_{i=1}^\iy |u_i|^2\la_i)^\f{1}{2}$, and
	
	(ii) $
		u\in H^2(\Om;w)\Lra \sum_{i=1}^\iy |u_i|^2\la_i^2<\iy$, 
	and
	\begin{equation*}
		\mcA u=\sum_{i=1}^\iy u_i\la_i\Phi_i \mbox{ and } \|\mcA u\|_{L^2(\Om)}=\left(\sum_{i=1}^\iy |u_i|^2\la_i^2\right)^\f{1}{2}.
	\end{equation*}
\end{lemma}

\begin{proof}
We prove this lemma by the following steps. 

{\it Step 1}. We prove (i). 

By \eqref{04.07.1}, for any $j,k\in\N^*$, we have
\begin{equation*}
(\Phi_j,\Phi_k)_{H_0^1(\Omega;w)}	=a(\Phi_j,\Phi_k) =\lambda_j(\Phi_j,\Phi_k)_{L^2(\Omega)}	=\lambda_j\delta_{jk}.
\end{equation*}
Therefore,
\begin{equation*}
\left(	\lambda_j^{-\frac12}\Phi_j, \lambda_k^{-\frac12}\Phi_k	\right)_{H_0^1(\Omega;w)}	=\delta_{jk},
\end{equation*}
and hence $\{\lambda_k^{-\frac12}\Phi_k\}_{k=1}^{\infty}$ is an orthonormal system in $H_0^1(\Omega;w)$.

It remains to prove its completeness. Let $u\in H_0^1(\Omega;w)$ satisfy
\begin{equation*}
\left(	u,\lambda_k^{-\frac12}\Phi_k \right)_{H_0^1(\Omega;w)}=0, \qquad k\in\N^*.
\end{equation*}
Since $\lambda_k>0$, it follows that
\begin{equation*}
	0=(u,\Phi_k)_{H_0^1(\Omega;w)} =a(u,\Phi_k).
\end{equation*}

By Hermitian symmetry, the definition of the operator $\mathcal A$ associated with the sesquilinear form $a$, and \eqref{04.07.1}, we have
\begin{equation*}
a(u,\Phi_k) =(u,\mathcal A\Phi_k)_{L^2(\Omega)} =\lambda_k(u,\Phi_k)_{L^2(\Omega)}.
\end{equation*}
	Thus
\begin{equation*}
(u,\Phi_k)_{L^2(\Omega)}=0,	\qquad k\in\N^*.
\end{equation*}
Since $\{\Phi_k\}_{k=1}^{\infty}$ is an orthonormal basis of the complex Hilbert space $L^2(\Omega)$, we conclude that $u=0$. Hence $\{\lambda_k^{-\frac12}\Phi_k\}_{k=1}^{\infty}$ is an orthonormal basis of $H_0^1(\Omega;w)$.
	
From above, since $u\in H_0^1(\Om;w)$,  we have 
	\begin{equation*}
		\nabla u=\sum_{i=1}^\iy e_i\la_i^{-\f{1}{2}}\nabla\Phi_i, \mbox{ and } \|u\|_{H_0^1(\Om;w)}^2=\sum_{i=1}^\iy |e_i|^2,
	\end{equation*}
	and 
	\begin{equation*}
		e_i=\int_\Om w \nabla u\cdot \la_i^{-\f{1}{2}}\nabla\Phi_i \df x=\la_i^{-\f{1}{2}}\int_\Om w\nabla u\cdot \nabla \Phi_i \df x=\la_i^\f{1}{2}\int_\Om u\Phi_i\df x=\la_i^\f{1}{2}u_i, 
	\end{equation*} 
	hence,  $\nabla u=\sum_{i=1}^\iy u_i\nabla\Phi_i$. Moreover, $\|u\|_{H_0^1(\Om;w)}=(\sum_{i=1}^\iy \la_i|u_i|^2)^\f{1}{2}$.

	{\it Step 2}.  We prove (ii).

Let $u\in H^2(\Om;w)$. Take $\vp^n=\sum_{i=1}^n u_i\Phi_i\in H_0^1(\Om;w)\cap H^2(\Om;w)$ for each $n\in\N^*$, then from
\begin{equation*}
	\begin{split}
		( \mcA u, \mcA\vp^n)_{L^2(\Om)}
		&= \sum_{i=1}^n \ol u_i\la_i \int_\Om (\mcA u)\Phi_i\df x= \sum_{i=1}^n \ol u_i\la_i\int_\Om u\mcA\Phi_i\df x\\
		&= \sum_{i=1}^n \ol u_i\la_i^2\int_\Om u\Phi_i\df x=\sum_{i=1}^n |u_i|^2\la_i^2
	\end{split}
\end{equation*}
and $\|\mcA\vp^n\|_{L^2(\Om)}^2=\sum_{i=1}^n |u_i|^2\la_i^2$ we obtain
\begin{equation*}
	\sum_{i=1}^n|u_i|^2\la_i^2\leq \|\mcA u\|_{L^2(\Om)}^2
\end{equation*}
for all $n\in\N^*$ by the Cauchy inequality. This implies that $\sum_{i=1}^\iy |u_i|^2\la_i^2\leq \|\mcA u\|_{L^2(\Om)}^2<\iy$.

Let $\sum_{i=1}^\iy |u_i|^2\la_i^2<\iy$. For each  $\vp\in C_0^\iy(\Om)$,  we have
\begin{equation*}
	\begin{split}
		( \mcA u, \vp)_{L^2(\Om)}
		&= \int_\Om w\nabla u\cdot \ol{\nabla \vp}\df x\\
		&= \sum_{i=1}^\iy u_i \int_\Om w\nabla\Phi_i\cdot \ol{\nabla\vp}\df x= \sum_{i=1}^\iy u_i\la_i\int_\Om \Phi_i\ol \vp\df x.
	\end{split}
\end{equation*}
Note that
\begin{equation*}
	\int_\Om \left|\sum_{i=1}^n u_i\la_i\Phi_i\right|^2\df x=\sum_{i=1}^n |u_i|^2\la_i^2\leq \sum_{i=1}^\iy |u_i|^2\la_i^2<\iy \mbox{ for all } n\in\N^*,
\end{equation*}
i.e., $\sum_{i=1}^\iy u_i\la_i\Phi_i\in L^2(\Om)$, hence
\begin{equation*}
		(\mcA u, \vp)_{L^2(\Om)}=\left(\sum_{i=1}^\iy u_i\la_i\Phi_i, \vp\right)_{L^2(\Om)}.
	\end{equation*}
	This implies that $\|\mcA u\|^2_{L^2(\Om)}\leq \sum_{i=1}^\iy |u_i|^2\la_i^2$ and $\mcA u=\sum_{i=1}^\iy u_i\la_i\Phi_i$. This completes the proof of  the lemma. 
\end{proof}

\begin{corollary}\label{04.09.C1}
	Let $s\in [-1,+\iy)$. Then the operator 
	\begin{equation*}
		\mcA: D(\mcA^{s+1})\ra D(\mcA^s)
	\end{equation*}
	is an isometric linear isomorphism. Moreover, if $u=\sum_{n=1}^\iy u_n\Phi_n\in D(\mcA^{s+1})$, then 
	\begin{equation*}
		\mcA u=\sum_{n=1}^\iy \la_n u_n\Phi_n. 
	\end{equation*}
\end{corollary}

\begin{proof}
	This is a conclusion from Lemma \ref{06.28.L1} and Definition \ref{04.07.D2}. 
\end{proof}

\begin{remark}\label{04.09.R1}
	(1) Some similar conclusions have been considered by Definitions 2.18 and 2.19 in \cite{Yang}, and Remarks 2.20 and 2.21, and Lemma 2.22. 
	
	(2) If   the operator $\mcA$ is the Laplace $-\De$, then $C_0^\iy(\Om)\s D(\mcA^{\iy})$. But this does not hold anymore since there is a term $x^\be \nabla \vp$ in $\mcA^n\vp$ with $|\be|=n(\al-2)+1$ for $n\in\N^*$. However, for $N\geq 2$, we also have $C_0^\iy(\Om)\s D(\mcA)$. 	If $\al\in (0,\f{2}{3})$ and $N=2$, then $C_0^\iy(\Om)\not\subset D(\mcA^\f{3}{2})$.  This highlights a fundamental difference between degenerate and non-degenerate partial differential operators.
	
	(3) From Definition \ref{04.07.D2} and Lemma \ref{06.28.L1} and Corollary \ref{04.09.C1}, we get $D(\mcA^\f{1}{2})=H_0^1(\Om;w)$ and $D(\mcA^{-\f{1}{2}})=H^{-1}(\Om;w)$. 
\end{remark}

\subsection{Weak solution}

\begin{definition}\label{04.07.D1}
Let $\vp^0\in H_0^1(\Om;w)$ and $f\in L^2(0,T; H_0^1(\Om;w))$. We call 
\begin{equation*}
	\begin{split}
		\vp\in L^2(0,T; H_0^1(\Om;w))\cap H^1(0,T;H^{-1}(\Om;w))
	\end{split}
\end{equation*}
a weak solution to the equation \eqref{12.14.1} with respect to $(\vp^0,f)$, provided
	
	(i) for all $v\in H_0^1(\Om;w)$, and for a.e.~$t\in [0,T]$, we have
	\begin{equation*}
		(\im \pt_t\vp, v)_{H^{-1}(\Om;w), H_0^1(\Om;w)}+(\vp, v)_{H_0^1(\Om;w)}=(f, v)_{L^2(\Om)}, 
	\end{equation*}
	
	(ii) $\vp(0)=\vp^0$. 
\end{definition}

\begin{remark}
	We note that the dual pair $( \im \pt_t\vp, v)_{H^{-1}(\Om;w), H_0^1(\Om;w)}$ in Definition \ref{04.07.D1} (i) is complex dual pair, rather than a real dual pair. 
\end{remark}

\begin{lemma}\label{11.30.L2}
	Let $(\mcM,g)$ be a $C^m$-Riemannian manifold with compact boundary $\pt \mcM$. Then there exists a $C^{m-1}$-vector field $\bf{n}$ such that 
	\begin{equation*}
		\bf{n}(x)=\nu(x), \ x\in\pt \mcM, \mbox{ and } |\bf{n}(x)|\leq 1, \ x\in \mcM, 
	\end{equation*}
	where $\nu$ is the unit outward normal vector to $\pt \mcM$. 
\end{lemma}

\begin{theorem}\label{04.07.T1}
	Let $\vp^0\in H_0^1(\Om;w)$ and $f\in L^2(0,T; H_0^1(\Om;w))$. Then there exists a unique weak solution
	\begin{equation*}
		\vp\in L^2(0,T; H_0^1(\Om;w))\cap H^1(0,T;H^{-1}(\Om;w))
	\end{equation*}
	to the equation \eqref{12.14.1} with respect to $(\vp^0,f)$, 
	and 
	\begin{equation}\label{04.07.3}
		\begin{split}
			&\esssup_{t\in [0,T]}\|\vp(t)\|_{H_0^1(\Om;w)}+\|\pt_t\vp\|_{L^2(0,T; H^{-1}(\Om;w))}+\left\|\f{\pt\vp}{\pt\nu}\right\|_{L^2(\pt Q)}\\
			&\leq C\left(\|\vp^0\|_{H_0^1(\Om;w)}+\|f\|_{L^2(0,T; H_0^1(\Om;w))}\right), 
		\end{split}
	\end{equation}
	where the positive constant $C$ depends only on $\al, N$ and $M$ and $T$. And if $f=0$, then 
	\begin{equation}\label{04.10.13}
		\|\vp(t)\|_{L^2(\Om)}=\|\vp^0\|_{L^2(\Om)}, \mbox{ for all } t\in [0,T].
	\end{equation}
	
	Moreover, if we further assume $\vp^0\in D(\mcA)$ and $f\in H^1(0,T; H_0^1(\Om;w))$, then the weak solution 
	\begin{equation*}
		\vp\in L^2(0,T; D(\mcA))\cap H^1(0,T; L^2(\Om))
	\end{equation*}
	to the equation \eqref{12.14.1} with respect to $(\vp^0,f)$, 
	and 
	\begin{equation}\label{04.07.4}
		\begin{split}
			&\esssup_{t\in [0,T]}\left(\|\vp(t)\|_{D(\mcA)}+\|\vp(t)\|_{H^2(\Om-B_{3R})}+\|\pt_t\vp\|_{L^2(Q)}\right)\\
			&\leq C\left(\|\vp^0\|_{D(\mcA)}+\|f\|_{H^1(0,T; H_0^1(\Om;w))}\right), 
		\end{split}
	\end{equation}
	and if $\vp^0\in D(\mcA^\f{3}{2})$ and $f=0$ then 
	\begin{equation}\label{04.11.12}
		\left\|\f{\pt\vp}{\pt\nu}\right\|_{H^1(0,T; L^2(\pt\Om))}\leq C \|\vp^0\|_{D(\mcA^\f{3}{2})}, 
	\end{equation} 
	where the positive constant $C$ depends only on $\al, \Om-B_{R}$ and $T$. 
\end{theorem}

\begin{proof}
We prove this theorem by the following steps. 
	
{\it Step 1}. We use Galerkin method. We assume $\vp^0\in H_0^1(\Om;w)$ and $f\in L^2(0,T; H_0^1(\Om;w))$. 
	
We denote 
\begin{equation}\label{04.07.5}
	\begin{split} 
	\vp^m(x,t)
	&=\sum_{n=1}^m\vp_n^m(t)\Phi_n(x),\\ f^m(x,t)
	&=\sum_{n=1}^m f_n^m(t)\Phi_n(x) \mbox{ with } f_n^m=(f,\Phi_n)_{L^2(\Om)} \mbox{ for } n=1,\cdots, m, 
	\end{split} 
\end{equation}
where $\vp_n^m, n=1,\cdots, m$ are the solution of the following ordinary differential equations
\begin{equation}\label{04.07.6}
	\begin{split}
		\im \pt_t\vp_n^m+\la_n\vp_n^m=f_n^m, \ n=1,\cdots, m 
	\end{split}
\end{equation}
with initial data
\begin{equation}\label{04.07.7}
	\vp_n^m(0)=(\vp^0,\Phi_n)_{L^2(\Om)}, \  n=1,\cdots, m,  
\end{equation}
where $\pt_t\vp_n^m=\f{\df\vp_n^m}{\df t}, n=1,\cdots, m$.  It is obvious that there exists a unique solution to \eqref{04.07.6} and \eqref{04.07.7}. 	Note that \eqref{04.07.6} is equivalent to the following equation
\begin{equation}\label{04.07.8}
	\im (\pt_t\vp^m, \Phi_n)_{L^2(\Om)}+(\vp^m, \Phi_n)_{H_0^1(\Om;w)}=( f^m, \Phi_n)_{L^2(\Om)}, \ n=1,\cdots, m. 
\end{equation}
	
{\it Step 2}. Energy estimate. 

Multiplying \eqref{04.07.6} by $\ol{\vp_n^m(t)}$, summing over $n=1,\cdots, m$, then 
\begin{equation}\label{04.07.9}
	\begin{split}
		\im \sum_{n=1}^m [\pt_t\vp_n^m(t)]\ol{\vp_n^m(t)}+\sum_{n=1}^m \la_n \vp_n^m(t) \ol{\vp_n^m(t)}=\sum_{n=1}^m f_n^m(t)\ol{\vp_n^m(t)}, 
	\end{split}
\end{equation}
and 
\begin{equation}\label{04.07.10}
	\begin{split}
		-\im \sum_{n=1}^m \ol{\pt_t\vp_n^m(t)}\vp_n^m(t)+\sum_{n=1}^m \la_n  \ol{\vp_n^m(t)}\vp_n^m(t)=\sum_{n=1}^m \ol{f_n^m(t)}\vp_n^m(t)
	\end{split}
\end{equation}
by taking conjugate to \eqref{04.07.9}. 
Hence 
\begin{equation}\label{04.07.12}
	\begin{split}
		\im \sum_{n=1}^m \f{\df}{\df t}\left(\vp_n^m\ol{\vp_n^m}\right)=\im 2\Image \sum_{n=1}^m f_n^m(t)\ol{\vp_n^m(t)}
	\end{split}
\end{equation}
by \eqref{04.07.9} minus \eqref{04.07.10}.   Note that \eqref{04.07.12} is 
\begin{equation}\label{04.07.11}
	\begin{split}
		\f{\df}{\df t} \|\vp^m\|_{L^2(\Om)}^2=2\Image (f^m, \vp^m)_{L^2(\Om)}\leq \|f^m\|_{L^2(\Om)}^2+\|\vp^m\|_{L^2(\Om)}^2, 
	\end{split}
\end{equation}
which shows that 
\begin{equation*}
\|\vp^m(t)\|_{L^2(\Om)}^2\leq e^t\left(\|\vp^m(0)\|_{L^2(\Om)}^2+\int_0^t\|f^m(t)\|_{L^2(\Om)}^2\df t\right)
\end{equation*}
by integrating on $[0,t]$ and applying Gronwall's inequality. Therefore, 
\begin{equation}\label{04.07.13}
\esssup_{t\in [0,T]}\|\vp^m(t)\|_{L^2(\Om)}^2\leq e^T\left(\|\vp^0\|_{L^2(\Om)}^2+\|f\|_{L^2(Q)}^2\right) 
\end{equation}
for all $m\in\N^*$. 

Multiplying \eqref{04.07.6} by $\ol{\pt_t\vp_n^m(t)}$, summing $n=1,\cdots, m$, we get
\begin{equation}\label{04.08.1}
	\begin{split}
		\im \sum_{n=1}^m \left|\pt_t\vp_n^m\right|^2+\sum_{n=1}^m \la_n \vp_n^m\ol{\pt_t\vp_n^m}=\sum_{n=1}^m f_n^m\ol{\pt_t\vp_n^m}, 
	\end{split}
\end{equation}
and 
\begin{equation}\label{04.08.2}
	\begin{split}
		-\im \sum_{n=1}^m \left|\pt_t\vp_n^m\right|^2+\sum_{n=1}^m\la_n \ol{\vp_n^m}\pt_t\vp_n^m=\sum_{n=1}^m \ol{f_n^m}\pt_t\vp_n^m
	\end{split}
\end{equation}
by taking conjugate to \eqref{04.08.1}. Then 
\begin{equation*}
	\begin{split}
		2\Real \sum_{n=1}^m \la_n \vp_n^m\ol{\pt_t\vp_n^m}=2\Real \sum_{n=1}^m f_n^m\ol{\pt_t\vp_n^m}
	\end{split}
\end{equation*}
by \eqref{04.08.1} plus \eqref{04.08.2}. That is 
\begin{equation}\label{04.08.3}
	\begin{split}
		\f{\df}{\df t}\|\vp^m\|_{H_0^1(\Om;w)}^2
		&=2( \vp^m, \pt_t\vp^m)_{H_0^1(\Om;w)}\\
		&=2( f, \pt_t\vp^m)_{L^2(\Om)}\leq \|f(t)\|_{H_0^{1}(\Om;w)}^2+\|\pt_t\vp^m\|_{H^{-1}(\Om;w)}^2
	\end{split}
\end{equation}
by Lemma \ref{06.28.L1} and Corollary \ref{04.09.C1}. 
Note that  from \eqref{04.07.6} and Lemma \ref{06.28.L1} and Corollary \ref{04.09.C1}, we get 
\begin{equation}\label{04.10.1}
	\begin{split}
		\|\pt_t\vp^m\|_{H^{-1}(\Om;w)}^2
		&=\sum_{n=1}^m\la_n^{-1}|\pt_t\vp_n^m(t)|^2\leq 2\sum_{n=1}^m \la_n|\vp_n^m(t)|^2+2\sum_{n=1}^m \la_n^{-1}|f_n^m(t)|^2\\
		&\leq 2\|\vp^m\|_{H_0^1(\Om;w)}^2+2\la_1^{-1}\|f^m(t)\|_{L^2(\Om)}^2.
	\end{split}
\end{equation}
Together with this, \eqref{04.10.7} and  \eqref{04.08.3}, we get 
\begin{equation*}
	\begin{split}
		\f{\df}{\df t}\|\vp^m\|_{H_0^1(\Om;w)}^2\leq C\|f(t)\|_{H_0^1(\Om;w)}^2+2\|\vp^m\|_{H_0^1(\Om;w)}^2, 
	\end{split}
\end{equation*}
where the positive constant $C$ depends only on $\al, N$ and $M$. Hence
\begin{equation*}
	\begin{split}
		\|\vp^m(t)\|_{H_0^1(\Om;w)}^2\leq Ce^{2t}\left(\|\vp^m(0)\|_{H_0^1(\Om;w)}^2+\int_0^t\|f(t)\|_{H_0^1(\Om;w)}^2\df t\right)
	\end{split}
\end{equation*}
by the Gronwall's inequality, 
where the positive constant $C$ depends only on $\al, N$ and $M$. This shows that 
\begin{equation}\label{04.10.3}
	\begin{split}
		\esssup_{t\in [0,T]}\|\vp^m\|_{H_0^1(\Om;w)}^2\leq Ce^{2T}\left(\|\vp^0\|_{H_0^1(\Om;w)}^2+\|f\|_{L^2(0,T; H_0^1(\Om;w))}^2\right), 
	\end{split}
\end{equation}
where the positive constant $C$ depends only on $\al, N$ and $M$. Together with this and   \eqref{04.10.1}, we get
\begin{equation}\label{04.10.2}
	\begin{split}
		\|\pt_t\vp^m\|_{L^2(0,T; H^{-1}(\Om;w))}^2\leq C\left(\|\vp^0\|_{H_0^1(\Om;w)}^2+\|f\|_{L^2(0,T; H_0^1(\Om;w))}^2\right), 
	\end{split}
\end{equation}
where the positive constant $C$ depends only on $\al, N, M$ and $T$. 

{\it Step 3}. Approximation, or convergence. 

From \eqref{04.07.13} and \eqref{04.10.3} and \eqref{04.10.2}, there exists a subsequence of $\{\vp^m\}_{m\in\N^*}$, still denoted by itself, and 
\begin{equation*}
	y\in L^2(0,T; H_0^1(\Om;w))\cap H^1(0,T; H^{-1}(\Om;w))
\end{equation*}
such that 
\begin{equation}\label{04.10.4}
	\begin{split}
		\vp^m
		&\ra y \mbox{ weak star in } L^\iy (0,T; H_0^1(\Om;w)),\\
		\pt_t\vp^m
		&\ra \pt_ty \mbox{ weakly in } L^2(0, T; H^{-1}(\Om;w)). 
	\end{split}
\end{equation}
Moreover, we obtain 
\begin{equation}\label{04.10.5}
	\begin{split}
		\|y\|_{L^\iy(0,T; H_0^1(\Om;w))}+\|\pt_ty\|_{L^2(0,T; H^{-1}(\Om;w))}\leq C\left(\|\vp^0\|_{H_0^1(\Om;w)}+\|f\|_{L^2(0,T; H_0^1(\Om;w))}\right), 
	\end{split}
\end{equation}
where the positive constant $C$ depends only on $\al, N$ and $M$ and $T$. From \cite[Theorem 3 in Chapter 5.9.2, p.~303]{Evans}, we get 
\begin{equation}\label{04.10.6}
	y\in C([0,T]; L^2(\Om)).
\end{equation}

{\it Step 4}. We prove $y$ is a weak  solution to the equation \eqref{12.14.1}. i.e., $y=\vp$. 

Now, for each $v\in C^1([0,T]; H_0^1(\Om;w))$ having the form
\begin{equation}\label{04.10.9}
	\begin{split}
		v(t)=\sum_{n=1}^kd_n^k(t)\Phi_n, \mbox{ where } \{d_n^k \}_{n=1}^k \mbox{ are given smooth functions}, 
	\end{split}
\end{equation}
then, for $m\geq k$, multiplying \eqref{04.07.8} by $\ol{d_n^k(t)}$, summing $n=1,\cdots, k$, integrating with respect to $t\in [0,T]$, note that $( \im \pt_t\vp^m, v)_{H^{-1}(\Om;w), H_0^1(\Om;w)}=( \im \pt_t\vp^m, v)_{L^2(\Om)}$, we get
\begin{equation}\label{04.10.10}
	\begin{split}
		\int_0^T( \im\pt_t\vp^m, v)_{H^{-1}(\Om;w), H_0^1(\Om;w)}\df t+\int_0^T(\vp^m, v)_{H_0^1(\Om;w)}\df t=\int_0^T(f^m, v)_{L^2(\Om)}\df t. 
	\end{split}
\end{equation}
This together with \eqref{04.10.4} and $f^m\ra f$ strongly in $L^2(0,T; H_0^1(\Om;w))$, we get 
\begin{equation}\label{04.10.8}
	\begin{split}
		\int_0^T(\im \pt_t y, v)_{H^{-1}(\Om;w), H_0^1(\Om;w)}\df t+\int_0^T( y, v)_{H_0^1(\Om;w)}\df t=\int_0^T( f, v)_{L^2(\Om)}\df t 
	\end{split}
\end{equation}
for all $v$ with the form \eqref{04.10.9}. Note that the set of the functions $v$ having the form \eqref{04.10.9} is dense in $L^2(0,T; H_0^1(\Om;w))$, then \eqref{04.10.8} holds for all $v\in L^2(0,T; H_0^1(\Om;w))$. This shows that 
\begin{equation*}
	(\im \pt_ty, v)_{H^{-1}(\Om;w), H_0^1(\Om;w)}+(y, v)_{H_0^1(\Om;w)}=( f, v)_{L^2(\Om)}
\end{equation*}
for all $v\in H_0^1(\Om;w)$. This proves Definition \ref{04.07.D1}. 

Now, we prove $y(0)=\vp^0$. 

From \eqref{04.10.8} and \eqref{04.10.6}, we get
\begin{equation*}
	\begin{split}
		-\int_0^T(\im y, \pt_tv)_{L^2(\Om)}\df t+\int_0^T(y, v)_{H_0^1(\Om;w)}\df t=\int_0^T( f,v)_{L^2(\Om)}\df t+( \im y(0), v(0))_{L^2(\Om)}
	\end{split}
\end{equation*}
for all $v\in C^1([0,T]; H_0^1(\Om;w))$ with $v(T)=0$.  From \eqref{04.10.10}, we get 
\begin{equation*}
	-\int_0^T(\im \vp^m, \pt_tv)_{L^2(\Om)}\df t+\int_0^T( \vp^m, v)_{H_0^1(\Om;w)}\df t=\int_0^T( f^m,v)_{L^2(\Om)}\df t+(\im \vp^m(0),v(0))_{L^2(\Om)},
\end{equation*} 
then 
\begin{equation*}
	-\int_0^T(\im y, \pt_tv)_{L^2(\Om)}\df t+\int_0^T( y, v)_{H_0^1(\Om;w)}\df t=\int_0^T (f, v)_{L^2(\Om)}\df t+(\im \vp^0, v(0))_{L^2(\Om)}
\end{equation*}
by \eqref{04.07.7} and \eqref{04.10.4}. These show that $y(0)=\vp^0$. 

{\it Step 5}. We prove the solution to equation \eqref{12.14.1} is unique. 

We only need to prove $\vp^0=f=0$, then the weak solution $\vp$ to the equation \eqref{12.14.1} is zero function. Indeed, taking $v= \vp$ in \eqref{04.10.8}, we get
\begin{equation}\label{04.10.11}
	\begin{split}
		\int_0^T(\im \pt_t\vp, \vp)_{H^{-1}(\Om;w), H_0^1(\Om;w)}\df t+\int_0^T(\vp, \vp)_{H_0^1(\Om;w)}\df t=0,
	\end{split}
\end{equation}
and 
\begin{equation}\label{04.10.12}
	-\int_0^T(\im \ol{\pt_t\vp}, \ol\vp)_{H^{-1}(\Om;w),H_0^1(\Om;w)}\df t+\int_0^T( \vp,\vp)_{H_0^1(\Om;w)}\df t=0
\end{equation}
by taking conjugate of \eqref{04.10.11}. These show that 
\begin{equation*}
	\begin{split}
		\|\vp(T)\|_{L^2(\Om)}^2-\|\vp(0)\|_{L^2(\Om)}^2=\int_0^T\f{\df}{\df t}\|\vp\|_{L^2(\Om)}^2\df t=0
	\end{split}
\end{equation*}
by \eqref{04.10.11} minus \eqref{04.10.12}.  Hence, $\varphi(T)=0$, which implies that $\varphi=0$ for all $t\in [0,T]$ by replacing $T$ with $t$ in \eqref{04.10.11} and \eqref{04.10.12}.

{\it Step 6}. Improved regularity. Assume $\vp^0\in D(\mcA)$ and  $f\in H^1(0,T; H_0^1(\Om;w))$.

From $f\in H^1(0,T; H_0^1(\Om;w))$ and \eqref{04.07.6}, we know that   $\pt_t\vp_n^m, n=1,\cdots,m$ is the solution of the following system 
\begin{equation}\label{04.11.1}
	\im \pt_{tt}\vp_n^m+\la_n\pt_t\vp_n^m=\pt_tf_n^m, \ n=1,\cdots, m
\end{equation}
with initial data 
\begin{equation}\label{04.11.2}
	\begin{split}
		\pt_t\vp_n^m(0)=-\im f_n^m(0)+\im \la_n \vp_n^m(0)=-\im f_n^m(0)+\im \la_n (\vp^0,\Phi_n)_{L^2(\Om)},
	\end{split}
\end{equation}
We note that $f_n^m(0)$ is well defined, since $f_n^m\in C[0,T]$ by $f\in H^1(0,T; H_0^1(\Om;w))$. 

Multiplying \eqref{04.11.1} by $\ol{\pt_t\vp_n^m}$, summing over $n=1,\cdots,m$, then 
\begin{equation}\label{04.11.3}
	\begin{split}
		\im \sum_{n=1}^m \left(\pt_{tt}\vp_n^m\right)\ol{\pt_t\vp_n^m}+\sum_{n=1}^m\la_n (\pt_t \vp_n^m)\ol{\pt_t\vp_n^m}=\sum_{n=1}^m \left(\pt_tf_n^m\right)\ol{\pt_t\vp_n^m}, 
	\end{split}
\end{equation}
and 
\begin{equation}\label{04.11.4}
	\begin{split}
		-\im \sum_{n=1}^m \ol{\pt_{tt}\vp_n^m}\pt_t\vp_n^m+\sum_{n=1}^m \la_n \ol{\pt_t\vp_n^m}\pt_t\vp_n^m=\sum_{n=1}^m \ol{\pt_tf_n^m}\pt_t\vp_n^m
	\end{split}
\end{equation}
by taking conjugate of \eqref{04.11.3}. These imply that 
\begin{equation*}
	\begin{split}
		\im \sum_{n=1}^m \f{\df}{\df t}(\pt_{t}\vp_n^m)\ol{\pt_t\vp_n^m}=\im 2\Image \sum_{n=1}^m (\pt_tf_n^m)\ol{\pt_t\vp_n^m}
	\end{split}
\end{equation*}
by \eqref{04.11.3} minus \eqref{04.11.4}, and then
\begin{equation}\label{04.11.5}
	\begin{split}
		\f{\df}{\df t}\|\pt_t\vp^m\|_{L^2(\Om)}^2=2\Image (\pt_tf^m, \pt_t\vp^m)_{L^2(\Om)}\leq \|\pt_tf^m\|_{L^2(\Om)}^2+\|\pt_t\vp^m\|_{L^2(\Om)}^2,
	\end{split}
\end{equation}
which shows that 
\begin{equation}\label{04.11.6}
	\begin{split}
		\|\pt_t\vp^m(t)\|_{L^2(\Om)}^2\leq e^t\left(\|\pt_t\vp^m(0)\|_{L^2(\Om)}^2+\int_0^t \|\pt_tf^m\|_{L^2(\Om)}^2\df t\right), \mbox{ for all } t\in [0,T]
	\end{split}
\end{equation}
by Gronwall's inequality. Hence, from \eqref{04.11.2} and \eqref{04.11.6} and Lemma \ref{06.28.L1} and Corollary \ref{04.09.C1}, we get
\begin{equation}\label{04.13.5}
	\begin{split}
		\esssup_{t\in [0,T]}\|\pt_t\vp^m(t)\|_{L^2(\Om)}^2
		&\leq e^T\left(\|f^m(0)\|_{L^2(\Om)}^2+\|\vp^0\|_{D(\mcA)}^2+\|\pt_tf^m\|_{L^2(Q)}^2\right)\\
		&\leq C\left(\|\vp^0\|_{D(\mcA)}^2+\|f\|_{H^1(0,T; L^2(\Om))}^2\right) 
	\end{split}
\end{equation}
by \cite[Theorem 2 in Chapter 5.9.2, p.~302]{Evans}, where the positive constant $C$ depends only on $T$. Together  with this and \eqref{04.10.4}, we obtain
\begin{equation}\label{04.11.7}
	\begin{split}
		\esssup_{t\in [0,T]}\|\pt_t\vp(t)\|^2_{L^2(\Om)}\leq C\left(\|\vp^0\|_{D(\mcA)}^2+\|f\|_{H^1(0,T; L^2(\Om))}^2\right), 
	\end{split}
\end{equation}
where the positive constant $C$ depends only on $T$.

Next, from \eqref{12.14.1} and \eqref{04.11.7} and \cite[Theorem 2 in Chapter 5.9.2, p.~302]{Evans}, we get
\begin{equation}\label{04.11.8}
	\begin{split}
		\esssup_{t\in [0,T]}\|\mcA \vp \|^2_{L^2(\Om)}\leq C\left(\|\vp^0\|_{D(\mcA)}^2+\|f\|_{H^1(0,T; L^2(\Om))}^2\right), 
	\end{split}
\end{equation}
where the positive constant $C$ depends only on $T$.

Finally, choosing 
\begin{equation}\label{04.11.11}
	\zeta\in C^{\infty}(\mathbb R^N),\quad  0\leq \zeta\leq 1
\end{equation} 
such that 
\begin{equation*}
	\zeta=0 \mbox{ on }B_{2R}, \quad \zeta=1 \mbox{ on }\Om-B_{3R}, \quad |\nabla\zeta|\leq CR^{-1} \mbox{ and }  \left|D^2\zeta\right|\leq CR^{-2} \mbox{ on }\R^N, 
\end{equation*}
where the positive constant $C$ are absolute. 
Then $\psi=\zeta \vp$ is the solution of the following equation
\begin{equation}\label{04.11.9}
	\mathcal A\psi=\zeta f-\im \zeta \pt_t\vp-2\nabla\zeta \cdot w\nabla\vp+\vp\mcA\zeta,\quad \psi=0 \mbox{ on } \partial{(\Om-\ol{B}_R)}. 
\end{equation}
By the standard elliptic regularity \cite[Theorem 1 (p.~327) in Chapter 6.3.1 and Theorem 4 (p.~334) in Chapter 6.3.2]{Evans}, we get
\begin{equation*}
		\|\psi(t)\|_{H^2(\Om-B_R(0))} \leq C\left(\|f(t)\|_{L^2(\Om)}+\|\pt_t\vp(t)\|_{L^2(\Om)}+\|\vp(t)\|_{H_0^1(\Om;w)}\right),
\end{equation*}
where the positive constant $C$ depends only on $\al$ and $\Om-B_R$. This shows that 
\begin{equation}\label{04.11.10}
	\begin{split}
		\esssup_{t\in [0,T]}\|\vp(t)\|_{H^2(\Om-B_{3R})}\leq C\left(\|\vp^0\|_{D(\mcA)}+\|f\|_{H^1(0,T; H_0^1(\Om;w))}\right)
	\end{split}
\end{equation}
by \eqref{04.10.5} and \eqref{04.11.7} and \cite[Theorem 2 in Chapter 5.9.2, p.~302]{Evans}, where the positive constant $C$ depends only on $\al$ and $\Om-B_R$ and $T$. 

{\it Step 7}. Hidden regularity. 
i.e., we shall show the three terms in \eqref{04.07.3} and \eqref{04.11.12}. 

We first prove \eqref{04.07.3}, as well as \eqref{04.11.12} follows by an analogous argument. 

By the density argument, i.e., $D(\mc A)$ is dense in $H_0^1(\Omega;{w})$ and $H^1(0,T;H_0^1(\Omega;w))$ is dense in $L^2(0,T;H_0^1(\Omega;w))$, we only need to show the case $\varphi^0 \in D(\mc A)$ and $f \in H^1(0,T;H_0^1(\Omega;w))$. 

Indeed, let $\bf{n}$ be defined in Lemma \ref{11.30.L2}, $\zeta$ be defined in \eqref{04.11.11}. Multiplying \eqref{12.14.1} by $(\zeta\bf{n})\cdot \nabla \ol\vp+\f{1}{2}\ol\vp\Div(\zeta \bf{n})$, integrating on $Q$, then 
\begin{equation}\label{04.11.13}
	A_1+A_2=A_3,
\end{equation}
where 
\begin{equation*}
	\begin{split}
		A_1
		&\equiv \im \iint_Q \pt_t\vp\left( (\zeta\bf{n})\cdot\nabla\ol \vp+\f{1}{2}\ol\vp\Div(\zeta\bf{n})\right)\df x\df t,\\
		A_2
		&\equiv -\iint_Q \left[\Div(w\nabla\vp)\right]\left(  (\zeta\bf{n})\cdot\nabla\ol \vp+\f{1}{2}\ol\vp\Div(\zeta\bf{n})\right)\df x\df t,\\
		A_3
		&\equiv \iint_Q f\left( (\zeta\bf{n})\cdot\nabla\ol \vp+\f{1}{2}\ol\vp\Div(\zeta\bf{n})\right)\df x\df t.
	\end{split}
\end{equation*}

Note that $\supp(\zeta\bf{n})\s \Om-B_R$, then we can integrate the terms $A_i, i=1,2,3$ by parts since the degenerate point $0\in\R^N$ has been cut-off.

From 
\begin{equation*}
	\begin{split}
		&2\im \iint_Q (\pt_t\vp)\left[(\zeta\bf{n})\cdot \nabla\ol\vp\right]\df x\df t-2\im \iint_Q (\pt_t\ol\vp)\left[(\zeta\bf{n})\cdot \nabla\vp\right]\df x\df t\\
		&=\im \int_\Om \vp(\zeta \bf{n})\cdot \nabla\ol\vp\df x\bigg|_{t=0}^{t=T}-\im \int_\Om \ol\vp(\zeta\bf{n})\cdot \nabla\vp\df x\bigg|_{t=0}^{t=T}\\
		&\hspace{4.5mm}+\im \iint_Q \vp(\pt_t\ol\vp) \Div(\zeta \bf{n})\df x\df t-\im \iint_Q (\pt_t\vp)\ol\vp \Div(\zeta \bf{n})\df x\df t, 
	\end{split}
\end{equation*}
we get
\begin{equation}\label{04.12.1}
	\begin{split}
		2\Real A_1
		&=\im\iint_Q (\pt_t\vp)\left[(\zeta\bf{n})\cdot \nabla\ol\vp\right]\df x\df t-\im\iint_Q (\pt_t\ol\vp)\left[(\zeta \bf{n})\cdot \nabla\vp\right]\df x\df t\\
		&\hspace{4.5mm}+\f{1}{2}\im \iint_Q (\pt_t\vp)\ol\vp \Div(\zeta \bf{n})\df x\df t-\f{1}{2}\im\iint_Q (\pt_t\ol\vp)\vp \Div(\zeta \bf{n})\df x\df t\\
		&=-\Image \int_\Om \vp(\zeta \bf{n})\cdot \nabla\ol\vp\df x\bigg|_{t=0}^{t=T}. 
	\end{split}
\end{equation}

From $\vp=0$ on $\pt Q$ we obtain 
\begin{equation*}
	\begin{split}
		&2\Real \iint_Q w\nabla\vp\cdot \left[D^2\ol\vp(\zeta \bf{n})\right]\df x\df t\\
		&=\iint_Q w\nabla \vp\cdot \left[D^2\ol\vp(\zeta \bf{n})\right]\df x\df t+\iint_Q w\nabla\ol\vp\cdot \left[D^2\vp(\zeta\bf{n})\right]\df x\df t\\
		&=\sum_{i,j=1}^N\iint_Q w \f{\pt}{\pt x_j}\left(\f{\pt\vp}{\pt x_i}\f{\pt\ol\vp}{\pt x_i}\right)\zeta \bf{n}_j\df x\df t\\
		&=\iint_{\pt Q} w|\nabla\vp|^2\zeta \bf{n}\cdot \nu\df S\df t-\iint_Q |\nabla\vp|^2\Div(w\zeta \bf{n})\df x\df t\\
		&=\iint_{\pt Q} w\left|\f{\pt \vp}{\pt \nu}\right|^2\df S\df t-\iint_Q |\nabla\vp|^2\Div(w\zeta \bf{n})\df x\df t, 
	\end{split}
\end{equation*}
we get
\begin{equation*}
	\begin{split}
		\Real \iint_Q w\nabla\vp\cdot \left[D^2\ol\vp(\zeta \bf{n})\right]\df x\df t=\f{1}{2}\iint_{\pt Q} w\left|\f{\pt \vp}{\pt \nu}\right|^2\df S\df t-\f{1}{2}\iint_Q |\nabla\vp|^2\Div(w\zeta \bf{n})\df x\df t.
	\end{split}
\end{equation*}
Then
\begin{equation}\label{04.12.2}
	\begin{split}
		\Real A_2
		&=-\f{1}{2}\iint_{\pt Q} w\left|\f{\pt\vp}{\pt\nu}\right|^2\df S\df t+\Real \iint_Q w\nabla \vp\cdot \left[D(\zeta \bf{n})\nabla\ol\vp\right]\df x\df t\\
		&\hspace{4.5mm}+\f{1}{2}\iint_Q |\nabla\vp|^2\zeta \bf{n}\cdot \nabla w\df x\df t+\f{1}{2}\Real \iint_Q \ol\vp w\nabla\vp\cdot \nabla\Div(\zeta \bf{n})\df x\df t. 
	\end{split}
\end{equation}

Together with \eqref{04.11.13} and  \eqref{04.12.1} and \eqref{04.12.2}, we get
\begin{equation}\label{04.12.3}
	\begin{split}
		\iint_{\pt Q}w\left|\f{\pt\vp}{\pt\nu}\right|^2\df S\df t\leq C\left(\|\vp^0\|^2_{H_0^1(\Om;w)}+\|f\|^2_{L^2(0,T; H_0^1(\Om;w))}\right), 
	\end{split}
\end{equation}
where the positive constant $C$ depends only on $\al, \Om-B_R$ and $T$. 

Let $\varphi^0 \in H_0^1(\Omega;w)$ and $f \in L^2(0,T;H_0^1(\Om;w))$. There exist two sequences $\{\varphi_k^0\}_{k \in \mathbb N^*} \subset D(\mathcal A)$ and $\{f_k\}_{k\in \mathbb N^*} \subset H^1(0,T;H_0^1(\Omega;w)) $, such that 
\begin{equation*}
\varphi	_k^0 \rightarrow \varphi^0 \mbox{ strongly in } H_0^1(\Omega;w), \qquad f_k \rightarrow f \mbox{ strongly in } L^2(0,T;H_0^1(\Omega;w)).
\end{equation*}

Let $\varphi_k$ be the weak solution of \eqref{12.14.1} with respect to $(\varphi_k^0,f_k)$ for all $k \in \mathbb N^*$. Applying \eqref{04.12.3} to $\varphi_k-\varphi_l$, we have
$$\iint _{\partial Q} w \left|\frac{\partial (\varphi_k - \varphi_l) }{\partial \nu}\right|^2 \df S \df t \leq C\left(\| \varphi_k^0 - \varphi_l^0\|^2_{H_0^1(\Om;w)} + \| f_k - f_l\|^2_{L^2(0,T;H_0^1(\Omega;w))}\right).$$
Since $0\in\Omega$ and $\partial\Omega$ is compact, we have $\min_{x\in\partial\Omega}|x|>0$. Thus, $w=|x|^\alpha$ is bounded below by a positive constant on $\partial\Omega$, and the preceding weighted estimate implies the Cauchy property in $L^2(\partial Q)$. Hence, $\{\frac{\partial \varphi_k}{\partial \nu}\}_{k\in\mathbb N^*}$ is a Cauchy sequence in $L^2(\partial Q)$. Therefore, there exists $z\in L^2(\partial Q)$ such that
\begin{equation}\label{26.10.07.02}
	\frac{\partial \varphi_k}{\partial \nu} \rightarrow z \mbox{ strongly in } L^2(\partial Q).
\end{equation}

Now, we prove $z =\frac{ \partial \varphi}{\partial \nu}$, where $\varphi$  is the weak solution of \eqref{12.14.1} with respect to $(\varphi^0, f)$.

Let $\zeta$ be defined as in \eqref{04.11.11}. For each $v \in C^{\infty}(\ol Q)$ satisfying $v(0) = v(T) = 0$, we have 
\begin{equation*}
\iint_{Q} 	\im \partial_t\varphi_k v \df x \df t - \iint_{Q} v \Div(|x|^\alpha \nabla \varphi_k) \df x \df t = \iint_{Q} f_k v \df x \df t.
\end{equation*}
Note that 
\begin{equation*}
\begin{split}
&\iint_{Q} v \Div(|x|^\alpha \nabla \varphi_k) \df x \df t \\
=& \iint_{Q} v \Div (|x|^{\alpha} \nabla (\zeta \varphi_k)) \df x \df t  + \iint_{Q} v \Div (|x|^{\alpha} \nabla ((1 - \zeta) \varphi_k)) \df x \df t  \\
=&\iint_{\partial Q} v |x|^{\alpha} \frac{\partial (\zeta \varphi_k)}{\partial \nu} \df S \df t - \iint_{Q} |x|^{\alpha} \nabla  (\zeta \varphi_k) \cdot \nabla v \df x \df t -  \iint_{Q} |x|^{\alpha} \nabla  ((1- \zeta) \varphi_k) \cdot \nabla v \df x \df t\\
=& \iint_{\partial Q} v |x|^{\alpha} \frac{\partial  \varphi_k}{\partial \nu} \df S \df t - \iint_{Q} |x|^{\alpha} \nabla  \varphi_k \cdot \nabla v \df x \df t ,
\end{split}
\end{equation*}
where we have used $\zeta=1$ in a neighborhood of $\pt\Om$, $\zeta\vp_k\in L^2(0,T;H^2(\Om-B_R))$, and $(1-\zeta)\vp_k\in L^2(0,T;H_0^1(\Om;w))$. Then, 
\begin{equation*}
-\iint_{Q} 	\im \varphi_k \partial_t v \df x \df t - \iint_{\partial Q} v |x|^{\alpha} \frac{\partial \varphi_k}{\partial \nu} \df S \df t + \iint_{Q} |x|^{\alpha} \nabla  \varphi_k \cdot \nabla v \df x \df t = \iint_{Q} f_k v \df x \df t.
\end{equation*}
Applying the continuous dependence estimate \eqref{04.10.5} to $\varphi_k-\varphi$, whose initial datum and source term are $\varphi_k^0-\varphi^0$ and $f_k-f$, respectively, yields $\varphi_k\to\varphi$ strongly in $L^2(0,T;H_0^1(\Omega;w))$. Passing to the limit as $k\to+\infty$, we obtain
\begin{equation*}
	-\iint_{Q} 	\im \varphi \partial_t v \df x \df t - \iint_{\partial Q} v |x|^{\alpha} z \df S \df t + \iint_{Q} |x|^{\alpha} \nabla  \varphi \cdot \nabla v \df x \df t = \iint_{Q} f v \df x \df t.
\end{equation*}
Therefore, by Green's identity, $z$ is the normal trace of $\varphi$, i.e., $z=\frac{\partial\varphi}{\partial\nu}$.
By the uniqueness of the normal trace, this limit is independent of the choice of the approximating sequences.

Now, let $\vp^0\in D(\mcA^\f{3}{2})$ and $f=0$. 

Since $\vp$ is the solution of \eqref{12.14.1}, i.e.,  $\vp=\sum_{n=1}^\iy e^{\im \la_n t}\vp_n^0\Phi_n$ with $\vp_n^0=(\vp^0,\Phi_n)_{L^2(\Om)}$ for all $n\in\N^*$, and 
\begin{equation*}
\|\varphi^0\|^2_{D(\mathcal A^{\frac 32})}	= \sum_{n=1}^{\infty} \lambda_n^3 |\varphi_n^0|^2 < \infty \ \Rightarrow \|\partial_t \varphi (t) \|^2_{D(\mathcal A^{\frac 12})} \leq \sum_{n=1}^{\infty} \lambda_n^3 |\varphi_n^0|^2 = \|\varphi^0\|^2_{D(\mathcal A^{\frac 32})} < \infty,
\end{equation*}
this implies $\partial_t \varphi(0) = \im \mathcal A \varphi^0$ and then $\partial_t \varphi $  is a weak solution of \eqref{12.14.1} with respect to $(\im \mcA\vp^0,0)$. Together with this and \eqref{04.12.3}, we get
\begin{equation}\label{26.09.18.02}
\begin{split}
\iint_{\pt Q}w\left|\partial_t \f{\pt \varphi}{\pt\nu}\right|^2\df S\df t=	\iint_{\pt Q}w\left|\f{\pt(\partial_t \varphi)}{\pt\nu}\right|^2\df S\df t\leq C \|\mathcal A\vp^0\|^2_{H_0^1(\Om;w)} = C\|\varphi^0\|^2_{D(\mathcal A^{\frac32})} .
\end{split}
\end{equation}
Together with \eqref{04.12.3} and \eqref{26.09.18.02}, we get \eqref{04.11.12}.

{\it Step 8}. We complete the proof of this theorem.

From Step 1 to Step 6, we get the existence of the weak solution of the equation \eqref{12.14.1}. \eqref{04.10.5} and \eqref{04.12.3} together yield \eqref{04.07.3}. Step 7  implies both \eqref{04.07.4} and  \eqref{04.11.12}. 
The results \eqref{04.10.13}   follow directly from Step 5.  We complete the proof of this theorem. 
\end{proof}

\begin{proposition}\label{04.12.P1}
	Let $\vp^0\in H_0^1(\Om;w)$ and $f\in L^2(0,T;H_0^1(\Om;w))$. Then 
	\begin{equation}\label{04.12.5}
		\vp\in L^2(0,T; H_0^1(\Om;w))\cap H^1(0,T; H^{-1}(\Om;w))
	\end{equation}
	is a weak solution of \eqref{12.14.1} with respect to $(\vp^0,f)$ if and only if 
	\begin{equation}\label{04.12.6}
		\vp \in L^2(0,T; H_0^1(\Om;w))
	\end{equation}
	satisfies 
	\begin{equation}\label{04.12.4}
		\begin{split}
			-\im \iint_Q \vp \pt_t\ol \psi\df x\df t+\iint_Q w\nabla \vp\cdot \nabla \ol \psi\df x\df t=\iint_Q f\ol \psi\df x\df t+\im \int_\Om \vp^0\ol \psi(0)\df x, 
		\end{split}
	\end{equation}
	where $\psi\in C^\iy(\ol Q)$ such that $\supp \psi(t)\s\Om$ for all $t\in [0,T]$ and $\psi(T)=0$. 
\end{proposition}

\begin{proof}
	We prove this proposition by the following steps. 
	
	{\it Step 1}. We prove the necessity. Let $\vp$ be defined as in \eqref{04.12.5}. 
	
	From \eqref{04.12.5}, we get \eqref{04.12.6}. From the assumption of $\psi$ and Definition \ref{04.07.D1} (i), we get
	\begin{equation*}
		\begin{split}
			\im (\pt_t\vp, \psi)_{H^{-1}(\Om;w),H_0^1(\Om;w)} + (\vp, \psi)_{H_0^1(\Om;w)} =(f,\psi)_{L^2(\Om)}
		\end{split}
	\end{equation*}
	for a.e.~$t\in [0,T]$, and 
	\begin{equation}\label{04.12.7}
		\vp\in C([0,T]; L^2(\Om)). 
	\end{equation}
	Integrating on $[0,T]$, we obtain 
	\begin{equation*}
		\im \int_0^T(\pt_t\vp, \psi)_{H^{-1}(\Om;w), H_0^1(\Om;w)} \df t+\int_0^T(\vp,\psi)_{H_0^1(\Om;w)}  \df t=\int_0^T(f,\psi)_{L^2(\Om)}  \df t. 
	\end{equation*}
	Integration by parts with respect to $t\in [0,T]$, from \eqref{04.12.7}, we have
	\begin{equation*}
		-\im \int_0^T(\vp,\partial_t \psi)_{L^2(\Om)}  \df t+\int_0^T(\vp,\psi)_{H_0^1(\Om;w)}\df t=\int_0^T(f,\psi)_{L^2(\Om)}\df t+\im (\vp(0),\psi(0))_{L^2(\Om)}. 
	\end{equation*}
	i.e., \eqref{04.12.4} holds. 
	
	{\it Step 2}. We prove the sufficiency. Let \eqref{04.12.6} and \eqref{04.12.4} hold.
	
	Let $\psi\in C_0^\iy(Q)$. 
	It is obvious that $f\in L^2(0,T; H^{-1}(\Om;w))$, and 
	\begin{equation*}
		 \mcA: L^2(0,T; H_0^1(\Om;w))\ra L^2(0,T; H^{-1}(\Om;w)) 
	\end{equation*}
	is an isometric linear isomorphism by Corollary \ref{04.09.C1}. 
	Then, from \eqref{04.12.4}, we get
	\begin{equation*}
		-\im \iint_Q \vp\pt_t\ol\psi\df x\df t+\iint_Q w\nabla\vp\cdot \nabla\ol\psi\df x\df t=\iint_Q f\ol\psi\df x\df t. 
	\end{equation*}
	i.e., 
	\begin{equation}\label{04.12.8}
		\im \int_0^T( \pt_t\vp, \psi)\df t+\int_0^T( \mcA\vp, \psi)_{H^{-1}(\Om;w), H_0^1(\Om;w)}\df t=\int_0^T( f, \psi)_{H^{-1}(\Om;w), H_0^1(\Om;w)}\df t, 
	\end{equation}
	this shows that $\pt_t\vp=-\im f+\im \mcA\vp\in L^2(0,T; H^{-1}(\Om;w))$. Hence, $\vp\in H^1(0,T; H^{-1}(\Om;w))$. 
	
	Now, we prove $\vp$ is a solution of \eqref{12.14.1}. 
	
	For fixed $t\in (0,T)$, taking $\de\in (0, \f{1}{4}\min \{t,T-t\})$, for each $\e\in (0,\de)$, choosing $\zeta=\zeta(t)\in C^\iy(\R),0\leq \zeta\leq 1$ such that 
	\begin{equation*}
		\zeta=1 \mbox{ on } (t-\de,t+\de), \quad \zeta=0 \mbox{ on } (-\iy, t-\de-\e)\cup (t+\de+\e,+\iy). 
	\end{equation*}
	For each $v\in C_0^\iy(\Om)$,   since $\zeta v\in C_0^\iy(Q)$, we get
	\begin{equation*}
		\begin{split}
			\im \int_0^T\zeta ( \pt_t\vp, v)_{H^{-1}(\Om;w), H_0^1(\Om;w)}\df t+\int_0^T\zeta (\vp, v)_{H_0^1(\Om;w)}\df t=\int_0^T\zeta (f,v)_{L^2(\Om)}\df t, 
		\end{split}
	\end{equation*}
	letting $\e\ra 0^+$, we obtain 
	\begin{equation*}
		\im \int_{t-\de}^{t+\de}( \pt_t\vp, v)_{H^{-1}(\Om;w), H_0^1(\Om;w)}\df t+\int_{t-\de}^{t+\de}(\vp, v)_{H_0^1(\Om;w)}\df t=\int_{t-\de}^{t+\de}(f,v)_{L^2(\Om)}\df t. 
	\end{equation*}
	From the Lebesgue point theorem, we get
	\begin{equation*}
		\im (\pt_t\vp, v)_{H^{-1}(\Om;w), H_0^1(\Om;w)}+(\vp,v)_{H_0^1(\Om;w)}=(f,v)_{L^2(\Om)}
	\end{equation*}
	for a.e.~$t\in [0,T]$. This implies that Definition \ref{04.07.D1} (i), since  $C_0^\iy(\Om)$ is dense in $H_0^1(\Om;w)$. 
	
	Let $v\in C_0^\iy(\Om)$. For each $\zeta=\zeta(t)\in C^\iy(\R)$ with $\zeta(T)=0$, from \eqref{04.12.4}, we get
	\begin{equation*}
		-\im\iint_Q \vp \zeta' \ol v\df x\df t+\iint_Q \zeta w\nabla\vp \cdot \nabla \ol v\df x\df t=\iint_Q \zeta f\ol v\df x\df t+\im\zeta(0)\int_\Om \vp(0)\ol v\df x.
	\end{equation*}
	Letting $v\ra \Phi_n\ (n\in\N^*)$ in $H_0^1(\Om;w)$, then
	\begin{equation*}
		-\im \int_0^T \zeta' \vp_n(t)\df t+\la_n\int_0^T \zeta \vp_n(t)\df t=\int_0^T \zeta f_n(t)\df t+\im \zeta(0)\vp_n^0, 
	\end{equation*}
	where $\vp_n(t)=(\vp,\Phi_n)_{L^2(\Om)}$ and $f_n(t)=(f,\Phi_n)_{L^2(\Om)}$ for all $n\in\N^*$. 
	This shows that $\vp_n(t)$ is a solution of the following ordinary differential equation
	\begin{equation*}
		\begin{cases}
			\im \f{\df}{\df t}\vp_n(t)+\la_n\vp_n(t)=f_n(t),& \mbox{for } t\in [0,T],\\
			\vp_n(0)=\vp_n^0. 
		\end{cases}
	\end{equation*}
	From Theorem \ref{04.07.T1} (see \eqref{04.10.4}) we get $\sum_{n=1}^m \vp_n(t)\Phi_n(x)$ converges weakly to the solution of \eqref{12.14.1} in $L^2(0,T; H_0^1(\Om;w))$ as $m\ra\iy$. Note that $\sum_{n=1}^m\vp_n(t)\Phi_n(x)\ra \sum_{n=1}^\iy \vp_n(t)\Phi_n(x)=\vp$ by \eqref{04.12.6} as $m\ra\iy$, hence $\vp$ is the weak solution of \eqref{12.14.1}. This shows that $\vp(0)=\vp^0$. i.e., Definition \ref{04.07.D1} (ii) holds. We complete the proof of this proposition. 
\end{proof}

\section{Approximation}\label{S3}
In this section, we construct a sequence of Schr\"odinger equations with uniformly elliptic spatial operators that approximate the degenerate Schr\"odinger equation \eqref{12.14.1}.

Here, we take $\e\in (0,R)$. We note that $w_{\epsilon} = w$ on $\Omega - B_R$.

\begin{lemma}\label{12.17.L1}
	There exists a $C^{2,1}$ function $\psi_\e: \mathbb{R}\to\mathbb{R}$ with $\e\in (0,R )$ such that
	\begin{equation*}
		\psi_\e(x)=|x| \quad \text{for } |x|\geq \e, \quad |x|\leq \psi_\e(x)\leq \e \quad \text{on } [-\e, \e],
	\end{equation*}
	and
	\begin{equation}\label{12.16.1}
		\psi_\e\geq \frac{3\e}{8} \quad \text{on } \mathbb{R}, \quad |\psi_\e'|\leq C   \text{ on } \mathbb{R}, \quad |\psi_\e''|\leq \frac{C}{\e},  \text{ and } |\psi_\e'''|\leq \frac{C}{\e^2}  \text{ on } \mathbb{R},
	\end{equation}
	where the constant $C>0$ are absolute.
\end{lemma}
\begin{proof}
	We consider the following polynomial function
	\begin{equation*}
		\psi_\e(x)=\sum_{i=0}^4a_ix^{4-i},
	\end{equation*}
	where $a_0, \cdots, a_4$ are constants to be determined. We impose the conditions
	\begin{equation*}
		\begin{split}
			\psi_\e(-\e)=\e, \quad \psi_\e(\e)=\e, \quad \psi_\e'(-\e)=-1, \quad \psi_\e'(\e)=1, \quad \psi_\e''(\pm \e)=0,
		\end{split}
	\end{equation*}
	which yield the solution
	\begin{equation}\label{12.20.1}
		\psi_\e(x)=\frac{3\e}{8}+\frac{3}{4\e}x^2-\frac{1}{8\e^3}x^4, \quad x\in[-\e,\e].
	\end{equation}
	This function satisfies the required properties.
\end{proof}
\begin{remark}\label{06.07.R1}
Let $x\in\mathbb{R}^N$, and consider the function $\psi_\e(x)=\psi_\e(|x|)$ defined in Lemma  \ref{12.17.L1}. We have (for all $i,j,k=1,\cdots, N$)
	\begin{equation}\label{12.18.1}
		\begin{split}
			\f{\pt \psi_\e}{\pt x_i}
			&=\frac{3}{2\e}x_i-\frac{1}{2\e^3}|x|^2x_i, \\
			\frac{\partial^2 \psi_\e}{\partial x_i\partial x_j}
			&=\delta_{ij}\left[\frac{3}{2\e}-\frac{1}{2\e^3}|x|^2\right]-\frac{1}{\e^3}x_ix_j, \\
			\f{\pt^3\psi_\e}{\pt x_i\pt x_j\pt x_k}
			&=-\f{1}{\e^3}\left(\de_{ij}x_k+\de_{ik}x_j+\de_{jk}x_i\right)
		\end{split}
	\end{equation}
	on $B_\e$,  and
	\begin{equation}\label{12.18.2}
		\begin{split}
			&\f{\pt \psi_\e}{\pt x_i}=|x|^{-1}x_i, \quad \frac{\partial^2\psi_\e}{\partial x_i\partial x_j}=\delta_{ij}|x|^{-1}-x_ix_j|x|^{-3}, \\ &\f{\pt^3\psi_\e}{\pt x_i\pt x_j\pt x_k}=-|x|^{-3}\left(\de_{ij}x_k+\de_{ik}x_j+\de_{jk}x_i\right)+3|x|^{-5}x_ix_jx_k
		\end{split}
	\end{equation}
on $\Omega-B_\e$. Here and in what follows, we denote $\de_{ij}, \de^{ij}$ and $\de_i^j$ by the Kronecker delta symbol, i.e., they are equal to $1$ when $i=j$ and equal to $0$ when $i\neq j$. 
\end{remark}

Denote
\begin{equation}\label{11.14.1}
	\begin{split}
		w_\e(x)= (\psi_\e(x))^\al=(\psi_\e(|x|))^\al, \quad x\in\Omega.
	\end{split}
\end{equation} 
Direct computation yields
\begin{equation}\label{11.25.1}
	\f{\pt w_\e}{\pt x_i}=\al \psi_\e^{\al-1}\f{\pt \psi_\e}{\pt x_i}=\al w_\e \psi_\e^{-1}\f{\pt\psi_\e}{\pt x_i}, \mbox{ and } \f{\pt w_\e^{-1}}{\pt x_i}=-\al \psi_\e^{-\al-1}\f{\pt\psi_\e}{\pt x_i}=-\al w_\e^{-1}\psi_\e^{-1}\f{\pt\psi_\e}{\pt x_i}.
\end{equation} 

Replace $w$ by $w_\e$ for each $\e\in (0,R )$, by the same argument as Section \ref{S2}, we can define the Hilbert spaces 
\begin{equation*}
	H_0^1(\Om;w_\e), \  H^1(\Om;w_\e), \mbox{ and } H^2(\Om;w_\e). 
\end{equation*}
Note that $H_0^1(\Om;w_\e)=H_0^1(\Om)$ and $H^2(\Om;w_\e)=H^2(\Om)$, then  $H^{-1}(\Om;w_\e)=H^{-1}(\Om)$. Denote 
\begin{equation*}
	\mcA_\e u=-\Div(w_\e\nabla u), \mbox{ with } D(\mcA_\e)=H^2(\Om;w_\e)\cap H_0^1(\Om;w_\e)=H^2(\Om)\cap H_0^1(\Om). 
\end{equation*}

\begin{lemma}\label{08.16.L2}
	Let $z\in H_0^1(\Omega;w_\e)$. Then,
	\begin{equation}\label{08.20.1}
		(N+\al-2)^2\int_\Om \psi_\e^{\al-2}|z|^2\df x\leq 4\int_\Om \psi_\e^\al |\nabla z|^2\df x.
	\end{equation}
	Moreover,   
	\begin{equation}\label{08.31.10}
		\int_\Om |z|^2\df x\leq \f{4M^{2-\al}}{(N+\al-2)^2}\int_\Om \psi_\e^\al |\nabla z|^2\df x.
	\end{equation}
\end{lemma}
\begin{proof}
Let $z \in H_0^1(\Om;w_\e)$ be a real function. Since $w_\e\in C^{2,1}(\overline{\Omega})$, we have
	\begin{equation*}
		\begin{split}
			2\int_{\Omega} \psi_\e^{\al-2}z(x\cdot\nabla z)\df x
			&=\int_\Omega \psi_\e^{\al-2}x\cdot \nabla z^2\df x=\int_\Omega \Div\left(\psi_\e^{\al-2}z^2x\right)\df x-\int_\Omega z^2\Div\left(\psi_\e^{\al-2}x\right)\df x\\
			&=-(N+\al-2)\int_\Omega z^2\psi_\e^{\al-2} \df x-(2-\al)\int_{B_\e}z^2\psi_\e^{\al-3}[\psi_\e-x\cdot\nabla\psi_\e]\df x.
		\end{split}
	\end{equation*}
	Then,
	\begin{equation*}
		\begin{split}
			(N+\al-2)\int_\Omega z^2\psi_\e^{\al-2}\df x
			&\leq -2\int_\Omega \psi_\e^{\al-2}z(x\cdot \nabla z)\df x=-2\int_\Omega \left(\psi_\e^\f{\al-2}{2}z\right)\left(\psi_\e^\f{\al-2}{2}x\cdot\nabla z\right)\df x\\
			&\leq 2\left(\int_\Omega \psi_\e^{\al-2}z^2\df x\right)^\frac{1}{2}\left(\int_\Omega \psi_\e^{\al-2}|x|^2|\nabla z|^2\df x\right)^\frac{1}{2}\\
			&\leq 2\left(\int_\Omega \psi_\e^{\al-2}z^2\df x\right)^\frac{1}{2}\left(\int_\Omega |\nabla z|^2\psi_\e^\al\df x\right)^\frac{1}{2}
		\end{split}
	\end{equation*}
by $\al\in (0,2)$, and $\psi_\e-x\cdot\nabla\psi_\e=\frac{3}{8\e^3}(\e^2-|x|^2)^2\geq 0$ on $B_\e$ (see Remark \ref{06.07.R1}), and Lemma \ref{12.17.L1}. This shows  \eqref{08.20.1}.

	Next, by \eqref{08.20.1} and  $(\f{3\e}{8})^\al\leq w_\e\leq M^\al$ in $\Omega$, we get  \eqref{08.31.10}. 
	
	Finally, let $z\in H_0^1(\Om;w)$ be a complex function, by the same argument as the proof of Lemma \ref{08.15.L1}, we get \eqref{08.20.1} and \eqref{08.31.10} for the complex cases. We complete the proof of this lemma. 
\end{proof}

The following lemma is similar to Lemma \ref{08.15.L4}; in fact $H_0^1(\Om;w_\e)=H_0^1(\Om)$. 

\begin{lemma}\label{11.27.L2}
	Let $\e>0$. Then the embedding $H_0^1(\Om;w_\e)\hra L^2(\Om)$ is compact. 
\end{lemma}

\begin{notation}\label{02.05.N1}
Denote $\mcA_\e u=-\Div(w_\e\nabla u)$ for $\e\in (0,R )$. From  Lemma \ref{08.16.L2} we obtain that the non-degenerate partial differential operator $\mcA_\e$ has discrete point spectrum 
	\begin{equation}\label{11.27.5}
		0<\la_1^\e\leq \la_2^\e\leq \la_3^\e\leq  \cdots \ra +\iy, 
	\end{equation}
	i.e., $\la_n^\e\ (n\in\N^*)$ satisfies the following equation
	\begin{equation*}
		\begin{cases}
			\mcA_\e \Phi_n^\e=\la_n^\e \Phi_n^\e, & \mbox{in }\Om, \\
			\Phi_n^\e=0, &\mbox{in }\pt\Om. 
		\end{cases}
	\end{equation*}
In what follows, we denote by $\Phi_n^\e(x)$ the $n$th eigenfunction of $\mcA_\e$ with respect to the eigenvalue $\la_n^\e\ (n\in\N^*)$, and $\{\Phi_n^\e\}_{n\in\N^*}$ is the orthonormal basis of $L^2(\Om)$. Moreover, $\{\Phi_n^\e\}_{n\in\N}$ is an orthogonal subset of $H_0^1(\Om;w_\e)$. (See Theorem 7 (pp. 728) in Appendix D in \cite{Evans}.) 
\end{notation}

\begin{definition}\label{04.13.D2}
	Let $s\geq 0$. We denote 
	\begin{equation*}
		D(\mcA_\e^s)=\left\{u\in L^2(\Om)\colon \sum_{n=1}^\iy (\la_n^\e)^{2s}|u_n|^2<+\iy \mbox{ for } u=\sum_{n=1}^\iy u_n\Phi_n^\e  \right\}. 
	\end{equation*}
	i.e., $u\in D(\mcA_\e^s)$ if and only if $((\la_n^\e)^su_n)_{n\in\N^*}\in \ell^2(\Cc)$. 
	
	Let $s\in [-1,0]$. We denote $D(\mcA_\e^s)$ as the dual space of $D(\mcA_\e^{-s})$. 
\end{definition}

\begin{lemma}\label{04.13.L1}
	Let $u=\sum_{i=1}^\iy u_i \Phi_i^\e\in H_0^1(\Om;w_\e)$ with $u_i=(u,\Phi_i^\e)_{L^2(\Om)}$ for all $i\in\N^*$. We have 
	
	(i) $\nabla u=\sum_{i=1}^\iy u_i\nabla \Phi_i^\e$ and $\|u\|_{H_0^1(\Om;w_\e)}=(\sum_{i=1}^\iy |u_i|^2\la_i^\e)^\f{1}{2}$;
	
	(ii) $
	u\in H^2(\Om;w_\e)\Lra \sum_{i=1}^\iy |u_i|^2(\la_i^\e)^2<\iy$, 
	and
	\begin{equation*}
		\mcA_\e u=\sum_{i=1}^\iy u_i\la_i^\e\Phi_i^\e \mbox{ and } \|\mcA_\e u\|_{L^2(\Om)}=\left(\sum_{i=1}^\iy |u_i|^2(\la_i^\e)^2\right)^\f{1}{2}.
	\end{equation*}
\end{lemma}

\begin{proof}
	The proof of this lemma is completely similar to Lemma \ref{06.28.L1}. 
\end{proof}

\begin{corollary}\label{04.13.C1}
	Let $s\in (-\iy,+\iy)$. Then, the operator 
	\begin{equation*}
		\mcA_\e: D(\mcA_\e^{s+1})\ra D(\mcA_\e^s)
	\end{equation*}
	is an isometric linear isomorphism. Moreover, if $u=\sum_{n=1}^\iy u_n\Phi_n^\e\in D(\mcA_\e^{s+1})$, then 
	\begin{equation*}
		\mcA_{\epsilon} u=\sum_{n=1}^\iy \la_n^\e u_n\Phi_n^\e. 
	\end{equation*}
\end{corollary}

\begin{proof}
	The proof of this corollary is completely analogous to that of Corollary \ref{04.09.C1}. 
\end{proof}

Now, we consider the following equation 
\begin{equation}\label{11.14.2}
	\begin{cases}
		\im \pt_{t}\vp_\e-\Div (w_\e\nabla \vp_\e)=f_\e, &\mbox{in } Q, \\
		\vp_\e=0, &\mbox{on }\pt Q, \\
		\vp_\e(0)=\vp_\e^0,  &\mbox{in }\Om, 
	\end{cases}
\end{equation}
where $w_\e$ is defined in \eqref{11.14.1}, $\vp_\e^0\in H_0^1(\Om;w_\e)$,  and  $f_\e\in L^2(0,T; H_0^1(\Om;w_\e))$. 

\begin{definition}\label{04.13.D1}
	We call 
	\begin{equation*}
		\vp_\e\in L^2(0, T; H_0^1(\Om;w_\e))\cap H^1(0,T; H^{-1}(\Om;w_\e))
	\end{equation*}
	is a weak solution of \eqref{11.14.2} with respect to $(\vp_\e^0, f_\e)$, provided
	
	(i) For each $v\in C_0^\iy(\Om)$ and a.e.~$t\in [0,T]$, 
	\begin{equation*}
		\begin{split}
			\im ( \pt_t\vp_\e, v)_{H^{-1}(\Om;w_\e), H_0^1(\Om;w_\e)}+(\vp_\e, v)_{H_0^1(\Om;w_\e)}=(f_\e,v)_{L^2(\Om)}, 
		\end{split}
	\end{equation*}
	
	(ii) $\vp_\e(0)=\vp_\e^0$. 
\end{definition}

\begin{theorem}\label{04.13.T1}
Let $\vp_\e^0\in H_0^1(\Om;w_\e)$ and $f_\e\in L^2(0,T; H_0^1(\Om;w_\e))$. Then there exists a unique weak solution
	\begin{equation*}
		\vp_\e\in L^2(0,T; H_0^1(\Om;w_\e))\cap H^1(0,T;H^{-1}(\Om;w_\e))
	\end{equation*}
to the equation \eqref{11.14.2} with respect to $(\vp_\e^0,f_\e)$,  and 
	\begin{equation}\label{04.13.1}
		\begin{split}
			&\esssup_{t\in [0,T]}\|\vp_\e(t)\|_{H_0^1(\Om;w_\e)}+\|\pt_t\vp_\e\|_{L^2(0,T; H^{-1}(\Om;w_\e))}+\left\|\f{\pt\vp_\e}{\pt\nu}\right\|_{L^2(\pt Q)}\\
			&\leq C\left(\|\vp_\e^0\|_{H_0^1(\Om;w_\e)}+\|f_\e\|_{L^2(0,T; H_0^1(\Om;w_\e))}\right), 
		\end{split}
	\end{equation}
where the positive constant $C$ depends only on $\al, N$ and $M$ and $T$. And if $f_\e=0$, then 
	\begin{equation}\label{04.13.2}
		\|\vp_\e(t)\|_{L^2(\Om)}=\|\vp_\e^0\|_{L^2(\Om)}, \mbox{ for all } t\in [0,T].
	\end{equation}
	
	Moreover, if we further assume $\vp_\e^0\in D(\mcA_\e)$ and $f_\e\in H^1(0,T; H_0^1(\Om;w_\e))$, then the weak solution 
	\begin{equation*}
		\vp_\e\in L^2(0,T; D(\mcA_\e))\cap H^1(0,T; L^2(\Om))
	\end{equation*}
	to the equation \eqref{11.14.2} with respect to $(\vp_\e^0,f_\e)$, 
	and 
	\begin{equation}\label{04.13.3}
		\begin{split}
			&\esssup_{t\in [0,T]}\left(\|\vp_\e(t)\|_{D(\mcA_\e)}+\|\vp_\e(t)\|_{H^2(\Om-B_{3R})}+\|\pt_t\vp_\e\|_{L^2(Q)}\right)\\
			&\leq C\left(\|\vp_\e^0\|_{D(\mcA_\e)}+\|f_\e\|_{H^1(0,T; H_0^1(\Om;w_\e))}\right), 
		\end{split}
	\end{equation}
	and if $\vp_\e^0\in D(\mcA_\e^\f{3}{2})$ and $f_\e=0$ then 
	\begin{equation}\label{04.13.4}
		\left\|\f{\pt\vp_\e}{\pt\nu}\right\|_{H^1(0,T; L^2(\pt\Om))}\leq C \|\vp_\e^0\|_{D(\mcA_\e^\f{3}{2})}, 
	\end{equation} 
	where the positive constant $C$ depends only on $\al, \Om-B_{R}$ and $T$. 
\end{theorem}

\begin{proof}
	The proof of this theorem is similar to that of Theorem \ref{04.07.T1}.
	
The estimates \eqref{04.13.1}, \eqref{04.13.3}, and \eqref{04.13.4} correspond to \eqref{04.07.3}, \eqref{04.07.4}, and \eqref{04.11.12}, respectively. Indeed, these estimates rely only on \eqref{04.10.5}, \eqref{04.11.7}, \eqref{04.11.8}, \eqref{04.11.10}, and \eqref{04.12.3}. Moreover, they depend only on Gronwall’s inequality, Lemma \ref{08.15.L1} (see \eqref{08.15.10}), Lemma \ref{06.28.L1}, Corollary \ref{04.09.C1}, and \cite[Theorem 2 in Chapter 5.9.2, p.~302]{Evans}.
	
Replacing Lemma \ref{08.15.L1} by Lemma \ref{08.16.L2} (Poincar\'e inequality), as well as replacing Lemma \ref{06.28.L1} and Corollary \ref{04.09.C1} by Lemma \ref{04.13.L1} and Corollary \ref{04.13.C1}, respectively, we obtain the estimates \eqref{04.13.1}, \eqref{04.13.3}, and \eqref{04.13.4} (the hidden-regularity holds in the region far away from $B_R$, where $w_{\epsilon} = w$). This completes the proof of the theorem.
\end{proof}

\begin{proposition}\label{04.13.P1}
	Let $\vp_\e^0\in H_0^1(\Om;w_\e)$ and $f_\e\in L^2(0,T;H_0^1(\Om;w_\e))$. Then 
	\begin{equation}\label{04.13.6}
		\vp_\e\in L^2(0,T; H_0^1(\Om;w_\e))\cap H^1(0,T; H^{-1}(\Om;w_\e))
	\end{equation}
	is a weak solution of \eqref{11.14.2} with respect to $(\vp_\e^0,f_\e)$ if and only if 
	\begin{equation}\label{04.13.7}
		\vp_\e \in L^2(0,T; H_0^1(\Om;w_\e))
	\end{equation}
	satisfies 
	\begin{equation}\label{04.13.8}
		\begin{split}
			-\im \iint_Q \vp_\e \pt_t\ol \psi\df x\df t+\iint_Q w_\e\nabla \vp_\e\cdot \nabla \ol \psi\df x\df t=\iint_Q f_\e\ol \psi\df x\df t+\im \int_\Om \vp_\e^0\ol \psi(0)\df x, 
		\end{split}
	\end{equation}
	where $\psi\in C^\iy(\ol Q)$ such that $\supp \psi(t)\s\Om$ for all $t\in [0,T]$ and $\psi(T)=0$. 
\end{proposition}

\begin{proof}
	The proof of this proposition is similar to that of Proposition \ref{04.12.P1}. 
\end{proof}

\begin{lemma}\label{04.14.A1}
Let $N\geq2$ and $\alpha\in(0,2)$. Then
\begin{equation*}
C_0^\infty(\Omega)\subset D(\mathcal A^{\frac 32}) \quad\Longleftrightarrow\quad 3\alpha+N>4.
\end{equation*}
Moreover, for every fixed $\epsilon\in(0,R)$,
\begin{equation*}
C_0^\infty(\Omega)\subset D(\mathcal A_\epsilon^{\frac 32}),
\end{equation*}
without the restriction $3\alpha+N>4$.
\end{lemma}

\begin{proof}
By the spectral definition and $D(\mathcal A^{\frac{1}{2}})=H_0^1(\Omega;w)$,
\begin{equation*}
u\in D(\mathcal A^{3/2}) \quad\Longleftrightarrow\quad u\in D(\mathcal A) \ \text{and}\  \mathcal A u\in H_0^1(\Omega;w).
\end{equation*}
For any $\phi\in C_0^\infty(\Omega)$, we have
\begin{equation*}
g:=\mathcal A\phi =-|x|^\alpha\Delta\phi -\alpha|x|^{\alpha-2}x\cdot\nabla\phi
\end{equation*}
in the distributional sense. Near the origin,
\begin{equation*}
	|g(x)|\leq C_\phi|x|^{\alpha-1}.
\end{equation*}
Since $N+2\alpha-2>0$, this implies $g\in L^2(\Omega)$, and hence $\phi\in D(\mathcal A)$. 

Suppose first that $3\alpha+N>4$. Differentiating away from the origin gives
\begin{equation*}
|\nabla g(x)|\leq C_\phi|x|^{\alpha-2}
\end{equation*}
near the origin. This gradient is locally integrable because $N+\alpha-2>0$. It also represents the distributional gradient of $g$: integration by parts outside $B_\rho$ produces an inner boundary term bounded by $C\rho^{N+\alpha-2}$, which tends to zero as $\rho\downarrow0$. Thus, for sufficiently small $\delta>0$,
\begin{equation*}
\int_{B_{2\delta}}w|\nabla g|^2\,\df x 	\leq C_\phi\int_0^{2\delta}r^{3\alpha+N-5}\,\df r	\leq C_\phi\delta^{3\alpha+N-4}.
\end{equation*}
Away from the origin, $g$ is smooth and has compact support in $\Omega$. Therefore, $g\in H^1(\Omega;w)$.
	
To show that $g\in H_0^1(\Omega;w)$, choose $\chi_\delta\in C^\infty(\mathbb R^N)$ satisfying
\begin{equation*}
0\leq\chi_\delta\leq1,\qquad \chi_\delta=0\ \text{on }B_\delta,\qquad \chi_\delta=1\ \text{outside }B_{2\delta},\qquad |\nabla\chi_\delta|\leq C\delta^{-1}.
\end{equation*}
	Then $\chi_\delta g\in C_0^\infty(\Omega)$, and
\begin{equation*}
\|(1-\chi_\delta)g\|_{L^2(\Omega)}^2 \leq C_\phi\delta^{N+2\alpha-2}.
\end{equation*}
	Moreover,
\begin{equation*}
\begin{split}
\int_\Omega w|\nabla((1-\chi_\delta)g)|^2\,\df x
&\leq 2\int_{B_{2\delta}}w|\nabla g|^2\,\df x +C\delta^{-2}	\int_{B_{2\delta}\setminus B_\delta}w|g|^2\,\df x\\
&\leq C_\phi\delta^{3\alpha+N-4}.
\end{split}
\end{equation*}
Both bounds tend to zero as $\delta\downarrow0$. Consequently, $\chi_\delta g\to g$ in $H^1(\Omega;w)$, so $g\in H_0^1(\Omega;w)$ and $\phi\in D(\mathcal A^{\frac 32})$.

Conversely, choose $\eta\in C_0^\infty(\Omega)$ with $\eta=1$ on a ball $B_\rho\subset\Omega$, and set $\phi(x)=\eta(x)x_1$. On $B_\rho\setminus\{0\}$,
\begin{equation*}
\mathcal A\phi=-\alpha|x|^{\alpha-2}x_1, \qquad \nabla(\mathcal A\phi) =-\alpha|x|^{\alpha-2}
\bigl(e_1+(\alpha-2)\omega_1\omega\bigr), \quad \omega=\frac{x}{|x|},
\end{equation*}
where $e_1=(1,0,\ldots,0)$. Hence
\begin{equation*}
\int_{B_\rho}w|\nabla(\mathcal A\phi)|^2\,\df x =c_{\alpha,N}\int_0^\rho r^{3\alpha+N-5}\,\df r,
\end{equation*}
with
\begin{equation*}
c_{\alpha,N} =\alpha^2|\mathbb S^{N-1}| \left(1-\frac{\alpha(2-\alpha)}{N}\right)>0.
\end{equation*}
Indeed, $0<\alpha(2-\alpha)\leq1$ and $N\geq2$. If $3\alpha+N\leq4$, the radial integral diverges. Thus $\mathcal A\phi\notin H_0^1(\Omega;w)$ and $\phi\notin D(\mathcal A^{\frac 32})$, proving necessity.
	
Finally, for each fixed $\epsilon\in(0,R)$, $w_\epsilon$ is strictly positive and belongs to $C^{2,1}(\overline\Omega)$. For every $\phi\in C_0^\infty(\Omega)$, we therefore have $\phi\in D(\mathcal A_\epsilon)$ and $\mathcal A_\epsilon\phi\in H_0^1(\Omega) =D(\mathcal A_\epsilon^{\frac 12})$. The spectral characterization yields $\phi\in D(\mathcal A_\epsilon^{\frac 32})$.
\end{proof}

\begin{theorem}\label{04.14.T1}
Let $\vp^0\in C_0^\iy(\Om)$ and $f\in L^2(0,T; H_0^1(\Om))$. Let $\vp, \vp_\e\ (\e\in (0,R))$ be the weak solutions with respect to $(\vp^0,f)$ and $(\vp^0_{\epsilon}=\vp^0,f_{\epsilon}=f)$, respectively.  Then there exists a sequence $\{{\epsilon_k} \}_{k \in \mathbb N^*}\s (0,R)$ with $\epsilon_k\downarrow0$ such that
\begin{equation}\label{04.14.1}
	\begin{split}
		\vp_{\e_k}
		&\ra \vp \mbox{ weakly in } L^2(0,T; H_0^1(\Om;w)). 
	\end{split}
\end{equation}
If we further assume $f\in H^1(0,T;H_0^1(\Omega;w))$, then, possibly after passing to a further subsequence, still denoted by ${\epsilon_k}$, we have
	\begin{equation}\label{04.14.2}
		\begin{split} 
			\pt_t\vp_{\e_k}
			&\ra \pt_t\vp \mbox{ weakly in }L^2(Q). 
		\end{split} 
	\end{equation}
Furthermore, if $f=0$ and $3\alpha+N>4$, then, possibly after passing to a further subsequence, still denoted by $\{\epsilon_k\}$, we have
	\begin{equation}\label{04.14.3}
		\begin{split}
			\vp_{\e_k} 
			&\ra \vp \mbox{ strongly in }L^2(Q), \\
			\f{\pt\vp_{\e_k}}{\pt\nu}
			&\ra \f{\pt\vp}{\pt\nu} \mbox{ strongly in } L^2(\pt Q). 
		\end{split}
	\end{equation}
\end{theorem}

\begin{proof}
We prove this theorem by the following steps. 
	
{\it Step 1}.  We prove \eqref{04.14.1}. 
	
Since $\vp^0\in C_0^\iy(\Om)$ and $f\in L^2(0,T; H_0^1(\Om))$, we get $\vp_\e\in L^2(0,T; H_0^1(\Om;w))$, and 
\begin{equation*}
	\iint_Q w\nabla\vp_\e\cdot \nabla\ol\vp_\e\df x\df t\leq \iint_Q w_\e\nabla\vp_\e\cdot\nabla\ol\vp_\e\df x\df t \mbox{ for all } \e\in (0,R). 
\end{equation*}
Together with these facts and \eqref{04.13.1} and $f\in L^2(0,T; H_0^1(\Om;w_\e))$ for all $\e\in (0,R)$, we get that there exists a sequence $\{\e_k\}_{k\in\N^*}\s(0,R)$ with $\e_k\downarrow 0$ and
\begin{equation}\label{04.14.5}
	y\in L^2(0,T; H_0^1(\Om;w))
\end{equation}
such that 
\begin{equation}\label{04.14.4}
\vp_{\e_k}\ra y\mbox{ weakly in } L^2(0,T; H_0^1(\Om;w)), \mbox{ as } k\ra\iy.	
\end{equation}
Now, we prove $y=\vp$. i.e., $y$ is a weak solution of the equation \eqref{12.14.1}. 
	
Indeed, for each $\psi\in C^\iy(\ol Q)$ with $\supp \psi(t)\s \Om$ for all $t\in [0,T]$ and $\psi(T)=0$, we get 
\begin{equation*}
-\im \iint_Q \vp_{\e_k} \pt_t\ol\psi\df x\df t+\iint_Q w_{\e_k} \nabla\vp_{\e_k}\cdot \nabla\ol\psi\df x\df t=\iint_Q f \ol\psi\df x\df t+\im \int_\Om \vp^0\ol\psi(0)\df x.
\end{equation*}
From \eqref{04.14.4} we get 
\begin{equation*}
	-\im\iint_Q \vp_{\e_k} \pt_t\ol\psi\df x\df t\ra	-\im \iint_Q y\pt_t\ol\psi \df x\df t,	\mbox{ as } k\ra\iy.
\end{equation*}
Hence, if 
\begin{equation}\label{04.14.6}
		\iint_Q w_{\e_k}\nabla\vp_{\e_k} \cdot \nabla\ol\psi\df x\df t\ra \iint_Q w\nabla y\cdot \nabla \ol\psi\df x\df t, \mbox{ as } k\ra\iy,  
\end{equation}
then $y=\vp$ by \eqref{04.12.4}. Note that 
\begin{equation*}
\iint_Q w_{\e_k} \nabla\vp_{\e_k}\cdot \nabla\ol\psi\df x\df t=\iint_Q w\nabla\vp_{\e_k}\cdot \nabla\ol\psi\df x\df t+\iint_Q (w_{\e_k}-w)\nabla\vp_{\e_k}\cdot \nabla\ol\psi\df x\df t, 
\end{equation*}
and from $w_{\e_k}=w$ on $\Omega - B_{\e_k}$ we get
\begin{equation*}
\iint_Q (w_{\e_k}-w)\nabla\vp_{\e_k}\cdot \nabla\ol\psi\df x\df t=\iint_{B_{\e_k} \ts (0,T)}(w_{\e_k}-w)\nabla\vp_{\e_k}\cdot \nabla\ol\psi\df x\df t, 
\end{equation*}
together with these and \eqref{04.14.4} and \eqref{04.14.6}, we only need to show 
\begin{equation}\label{04.14.7}
	\iint_{B_{\e_k} \ts (0,T)}(w_{\e_k}-w)\nabla\vp_{\e_k}\cdot \nabla\ol\psi\df x\df t \ra 0, \mbox{ as }  k \ra \iy 
\end{equation}
	
It is obvious that 
\begin{equation*}
\begin{split}
	&\left|\iint_{B_{\e_k} \ts (0,T)}w_{\e_k}\nabla\vp_{\e_k}\cdot \nabla\ol\psi\df x\df t\right|\\
	&\leq \iint_{B_{\e_k} \ts (0,T)} w_{\e_k} |\nabla\vp_{\e_k}||\nabla\psi|\df x\df t\leq \|\vp_{\e_k}\|_{L^2(0,T; H_0^1(\Om;w_{\e_k}))}\left(\iint_{B_{\e_k}\ts (0,T)} w_{\e_k} |\nabla\psi|^2\df x\df t\right)^\f{1}{2}\\
	&\leq C\sqrt{T}\left(\sup_{(x,t)\in Q}|\nabla\psi(x,t)|\right)\left(\|\vp^0\|_{H_0^1(\Om;w)}+\|f\|_{L^2(0,T; H_0^1(\Om;w))}\right){\e_k}^\f{N+\al}{2}
	\end{split}
\end{equation*}
by \eqref{12.16.1}, and similar to this we get
\begin{equation*}
	\begin{split}
	&\left|\iint_{B_{\e_k} \ts (0,T)}w\nabla\vp_{\e_k}\cdot \nabla\ol\psi\df x\df t\right|\\ 
	&\leq C\sqrt{T}\left(\sup_{(x,t)\in Q}|\nabla\psi(x,t)|\right)\left(\|\vp^0\|_{H_0^1(\Om;w)}+\|f\|_{L^2(0,T; H_0^1(\Om;w))}\right){\e_k}^\f{N+\al}{2}.
	\end{split}
\end{equation*}
These show that \eqref{04.14.7} holds. 
	
{\it Step 2}. We prove \eqref{04.14.2}. 

From \eqref{04.13.3} and \eqref{04.14.4}, possibly after passing to a further subsequence, still denoted by $\{\e_k\}$, we get
\begin{equation*}
	\pt_t\vp_{\e_k}\ra \pt_t\vp \mbox{ weakly in }L^2(Q).
\end{equation*}
This proves \eqref{04.14.2}. 

{\it Step 3}. We prove \eqref{04.14.3}. Assume that $f=0$ and $3\alpha+N>4$. Since
$\vp^0\in C_0^\infty(\Omega)$, Lemma \ref{04.14.A1} gives
\[
\vp^0\in D(\mathcal A^{\frac 32}) \cap D(\mathcal A_\epsilon^{\frac 32}), \qquad \epsilon\in(0,R).
\]	
Since the embeddings $H_0^1(\Om;w)\hra L^2(\Om)$ (see Lemma \ref{08.15.L4}) and 
\begin{equation*}
	\left\{z\in L^2(0,T; H_0^1(\Om;w))\colon \pt_t z\in L^2(Q)\right\}\hra L^2(Q)
\end{equation*}
are compact, from \eqref{04.14.4} and Step 2, possibly after passing to a further subsequence, still denoted by $\{\e_k\}$, we get the first part in \eqref{04.14.3}.

Since $\vp_{\e_k}\ ({\e_k}\in (0,R))$ is a weak solution of \eqref{11.14.2} with $f_{\e_k}=0$, we have
\begin{equation*}
	\vp^0=\sum_{n=1}^\iy \vp_n^0\Phi_n^{\e_k}, \mbox{ and } \vp_{\e_k}(t)=\sum_{n=1}^\iy e^{\im \la_n^{\e_k} t} \vp_n^0\Phi_n^{\e_k}. 
\end{equation*}
And hence $\pt_t \vp_{\e_k}(t)=\im\sum_{n=1}^\iy\la_n^{\e_k} e^{\im \la_n^{\e_k} t}\vp_n^0\Phi_n^{\e_k}\in H_0^1(\Om;w_{\e_k})$. This implies that $\pt_t\vp_{\e_k}$ is a weak solution of the following system 
\begin{equation*}
	\begin{cases}
		\im\pt_t\xi +\mcA_{\e_k} \xi=0, &\mbox{in }Q, \\
		\xi=0, &\mbox{on }\pt Q,\\
		\xi(0)=\im \mcA_{\e_k} \vp^0, &\mbox{in }\Om. 
	\end{cases}
\end{equation*}

Direct computation gives
\begin{equation*}
-\mcA_{\e_k}\vp^0 =\al\psi_{\e_k}^{\al-1} \nabla\vp^0\cdot \nabla\psi_{\e_k} +\psi_{\e_k}^{\al}\De\vp^0,
\end{equation*}
and
\begin{equation*}
\begin{split}
-\nabla(\mcA_{\e_k}\vp^0)
&=\al(\al-1)\psi_{\e_k}^{\al-2} (\nabla\vp^0\cdot\nabla\psi_{\e_k}) \nabla\psi_{\e_k} +\al\psi_{\e_k}^{\al-1} D^2\vp^0\nabla\psi_{\e_k}\\
&\quad+\al\psi_{\e_k}^{\al-1} D^2\psi_{\e_k}\nabla\vp^0 +\al\psi_{\e_k}^{\al-1} \De\vp^0\nabla\psi_{\e_k} +\psi_{\e_k}^{\al}\nabla(\De\vp^0).
\end{split}
\end{equation*}
Using \eqref{12.16.1} and \eqref{12.18.1}, together with $\vp^0\in C_0^\infty(\Omega)$, we obtain
\begin{equation*}
\int_{B_{\e_k}} \left( |\mcA_{\e_k}\vp^0|^2 +w_{\e_k}|\nabla(\mcA_{\e_k}\vp^0)|^2 \right)\,\df x \leq C\left( \e_k^{N+2\al-2}+\e_k^{N+3\al-4} \right),
\end{equation*}
where $C>0$ depends only on $\al$, $N$, $M$, and $\vp^0$. On $\Omega-B_{\e_k}$, we have $w_{\e_k}=w$ and $\mcA_{\e_k}\vp^0=\mcA\vp^0$. Moreover, Lemma \ref{04.14.A1} and $3\al+N>4$ give $\mcA\vp^0\in H_0^1(\Omega;w)$. Since $N+2\al-2>0$ and $N+3\al-4>0$, we obtain
\begin{equation*}
\begin{split}
\|\mcA_{\e_k}\vp^0\|_{L^2(\Om)}\leq C, \quad \left\|\mcA_{\e_k}\vp^0\right\|_{H_0^1(\Om;w_{\e_k})}\leq C,
\end{split}
\end{equation*}
where the positive constant $C$ depends only on $\al, N$ and $M$ and  $\vp^0$. Together with this and \eqref{04.13.1}, we get 
\begin{equation*}
\left\|\pt_t\f{\pt\vp_{\e_k}}{\pt\nu}\right\|_{L^2(\pt Q)}=\left\|\f{\pt(\pt_t\vp_{\e_k})}{\pt\nu}\right\|_{L^2(\pt Q)}\leq C\|\mcA_{\e_k}\vp^0\|_{H_0^1(\Om;w_{\e_k})}\leq C,
\end{equation*}
where $C$ is a positive constant independent of $k$.

Hence, possibly after passing to a further subsequence, still denoted by $\{\e_k\}$, we have
\begin{equation}\label{04.14.8}
	\begin{split} 
		\f{\pt\vp_{\e_k}}{\pt\nu}
		&\ra z \mbox{ weakly in } L^2(0,T; H^\f{1}{2}(\pt\Om)),\\
		\pt_t\f{\pt\vp_{\e_k}}{\pt\nu}
		&\ra \pt_t z\mbox{ weakly in }L^2(\pt Q)
	\end{split} 
\end{equation}
by \eqref{04.13.3} and the classical Sobolev trace theorem.
	
Note that for each $\eta\in C^\iy(\pt\Om)$ there exists $\wh \eta\in C^\iy(\R^N)$ with $\supp \wh \eta \s \R^N-B_R$ such that $\wh \eta=\eta$ on $\pt\Om$. Multiplying $\wh \eta$ on the both side of \eqref{11.14.2}, integrating on $Q$, then 
\begin{equation*}
	\begin{split}
		0
	&=\im \iint_Q (\pt_t\vp_{\e_k})\wh \eta \df x\df t-\iint_{\pt Q} w\f{\pt\vp_{\e_k}}{\pt \nu}\eta\df S\df t+\iint_Q w\nabla\vp_{\e_k}\cdot \nabla\wh\eta\df x\df t
		\end{split}
\end{equation*}
by \eqref{04.14.4} and \eqref{04.14.8}, and hence 
\begin{equation}\label{04.14.9}
	0 =\im \iint_Q (\pt_t\vp)\wh \eta \df x\df t-\iint_{\pt Q} wz\eta\df S\df t+\iint_Q w\nabla\vp\cdot \nabla\wh\eta\df x\df t
\end{equation}
by letting $k\ra\iy$. Multiplying $\wh \eta$ on the both sides of \eqref{12.14.1}, integrating on $Q$, then 
\begin{equation*}
	0 =\im \iint_Q (\pt_t\vp)\wh \eta \df x\df t-\iint_{\pt Q} w\f{\pt\vp}{\pt \nu}\eta\df S\df t+\iint_Q w\nabla\vp\cdot \nabla\wh\eta\df x\df t.
\end{equation*}
Together with this and \eqref{04.14.9} and the arbitrariness of $\eta\in C^\iy(\pt\Om)$, we get
\begin{equation}\label{04.14.10}
	z=\f{\pt\vp}{\pt \nu}. 
\end{equation}

Finally, since the embeddings  $H^\f{1}{2}(\pt\Om)\hra L^2(\pt\Om)$ and 
\begin{equation*}
	\left\{z\in L^2(0,T; H^\f{1}{2}(\pt\Om))\colon \pt_tz\in L^2(\pt Q)\right\}\hra L^2(\pt Q)
\end{equation*}
are compact, from \eqref{04.14.8} and \eqref{04.14.10}, we get
\begin{equation*}
	\f{\pt\vp_{\e_k}}{\pt\nu}\ra \f{\pt \vp}{\pt\nu} \mbox{ strongly in }L^2(\pt Q).
\end{equation*}
This shows the second part in \eqref{04.14.3}. We complete the proof of this theorem. 
\end{proof}

\section{Observability}\label{S4}
\subsection{Observability for the uniformly elliptic approximating equations}
In this subsection, we assume $\e\in (0,R)$, and $\vp_\e^0\in D(\mcA_\e^\f{3}{2})$ and  $f_\e$=0 in \eqref{11.14.2}. Then the weak solution $\vp_\e$ of equation \eqref{11.14.2} with respect to $(\vp_\e^0,f_\e)$ satisfies Theorem \ref{04.13.T1}, hence the integration by parts in the following is the same as the classical cases.

Let 
\begin{equation}\label{04.14.11}
	H=x.
\end{equation}
Denote
\begin{equation}\label{04.15.8}
	\Ga_0=\left\{x\in \pt\Om\colon H\cdot \nu>0\right\}, \quad \Sigma_0=\Ga_0\ts (0,T). 
\end{equation}
and \begin{equation}\label{04.15.7}
	\theta=\f{1}{2}\sup_{x\in\pt\Om} \left(|x|^\al H\cdot \nu\right).
\end{equation}

Multiplying $H\cdot \nabla \ol\vp_\e+\f{N}{2}\ol\vp_\e$ on the both sides of \eqref{11.14.2}, integration by parts, then, we have 
\begin{equation}\label{04.15.3}
	\begin{split}
		0
		&=\im \iint_Q (\pt_t\vp_\e)\left(H\cdot\nabla\ol\vp_\e+\f{N}{2}\ol\vp_\e\right)\df x\df t-\iint_Q [\Div(w_\e\nabla \vp_\e)]\left(H\cdot \nabla\ol \vp_\e+\f{N}{2}\ol\vp_\e\right)\df x\df t\\
		&\equiv B_1+B_2.
	\end{split}
\end{equation}

From 
\begin{equation*}
	\begin{split}
		&2\im \iint_Q (\pt_t\vp_\e)(H\cdot \nabla\ol\vp_\e)\df x\df t-2\im \iint_Q (\pt_t\ol\vp_\e)\left(H\cdot \nabla\vp_\e\right)\df x\df t\\
		&=\im \int_\Om \vp_\e(H\cdot \nabla\ol\vp_\e)\df x\bigg|_{t=0}^{t=T}-\im \int_\Om \ol\vp_\e (H\cdot \nabla\vp_\e)\df x\bigg|_{t=0}^{t=T}\\
		&\hspace{4.5mm}+\im N\iint_Q \vp_\e \pt_t\ol\vp_\e\df x\df t-\im N \iint_Q \ol\vp_\e \pt_t\vp_\e \df x\df t,
	\end{split}
\end{equation*}
we get
\begin{equation}\label{04.15.2}
	\begin{split}
		2\Real B_1
		&=\im \iint_Q (\pt_t\vp_\e)(H\cdot \nabla\ol\vp_\e)\df x\df t-\im \iint_Q (\pt_t\ol\vp_\e)\left(H\cdot \nabla\vp_\e\right)\df x\df t\\
		&\hspace{4.5mm}+\f{N}{2}\im \iint_Q (\pt_t\vp_\e)\ol\vp_\e\df x\df t-\f{N}{2}\im\iint_Q (\pt_t\ol\vp_\e)\vp_\e\df x\df t\\
		&=-\Image \int_\Om \vp_\e (H\cdot \nabla\ol\vp_\e)\df x\bigg|_{t=0}^{t=T}. 
	\end{split}
\end{equation}

From
\begin{equation*}
	\begin{split}
		&2\Real \iint_Q w_\e\nabla \vp_\e \cdot \nabla (H\cdot \nabla\ol \vp_\e)\df x\df t\\
		&=\iint_Q w_\e\nabla \vp_\e \cdot \nabla (H\cdot \nabla\ol \vp_\e)\df x\df t+\iint_Q w_\e\nabla \ol\vp_\e \cdot \nabla (H\cdot \nabla \vp_\e)\df x\df t\\
		&=2\iint_Q w_\e|\nabla\vp_\e|^2\df x\df t+\iint_Q w_\e H\cdot \nabla |\nabla\vp_\e|^2\df x\df t,
	\end{split}
\end{equation*}
and (by $\vp_\e=0$ on $\pt\Om$)
\begin{equation*}
	\begin{split}
		&2\Real \iint_Q [\Div(w_\e\nabla\vp_\e)]\frac{N}{2} \ol\vp_\e\df x\df t\\
		&=\f{N}{2}\iint_Q [\Div(w_\e\nabla\vp_\e)]\ol\vp_\e\df x\df t+\f{N}{2}\iint_Q [\Div(w_\e\nabla\ol\vp_\e)]\vp_\e\df x\df t=-N\iint_Q w_\e|\nabla\vp_\e|^2\df x\df t,
	\end{split}
\end{equation*}
we get
\begin{equation*}
	\begin{split}
		\Real B_2
		&=-\Real \iint_Q [\Div(w_\e\nabla \vp_\e)]\left(H\cdot\nabla\ol \vp_\e+\f{N}{2}\ol\vp_\e\right)\df x\df t\\
		&=-\Real \iint_{\partial Q} (w_\e\nabla \vp_\e\cdot\nu) (H\cdot\nabla\ol \vp_\e)\df S\df t+\Real \iint_Q w_\e\nabla \vp_\e \cdot \nabla (H\cdot \nabla\ol \vp_\e)\df x\df t\\
		&\hspace{4.5mm}+\f{N}{2}\iint_Q w_\e|\nabla\vp_\e|^2\df x\df t\\
		&=-\iint_{\partial Q} w_\e\left|\f{\pt\vp_\e}{\pt\nu}\right|^2H\cdot \nu\df S\df t+\f{N+2}{2}\iint_Q w_\e|\nabla\vp_\e|^2\df x\df t+\f{1}{2}\iint_Q w_\e H\cdot \nabla|\nabla\vp_\e|^2\df x\df t\\
		&=-\f{1}{2}\iint_{\partial Q} w_\e\left|\f{\pt\vp_\e}{\pt\nu}\right|^2H\cdot \nu\df S\df t+\iint_Q w_\e|\nabla\vp_\e|^2\df x\df t-\f{1}{2}\iint_Q|\nabla\vp_\e|^2H\cdot \nabla w_\e\df x\df t
	\end{split}
\end{equation*}
by $\vp_\e=0$ on $\pt\Om$. Note that from \eqref{12.18.1} and \eqref{12.18.2}, we get 
\begin{equation*}
	\begin{split}
		H\cdot \nabla w_\e
		&=
		\begin{cases}
			\al w_\e \psi_\e^{-1}\left(\f{3}{2\e}|x|^2-\f{1}{2\e^3}|x|^4\right), &\mbox{on }B_\e, \\
			\al w_\e, &\mbox{on }\Om-B_\e,
		\end{cases}\\
		&=\al w_\e-\f{3\al}{8\e^3}w_\e\psi_\e^{-1}\left(\e^2-|x|^2\right)^2, 
	\end{split}
\end{equation*}
Hence, we obtain
\begin{equation}\label{04.15.1}
	\begin{split}
		\Real B_2
		&=-\f{1}{2}\iint_{\partial Q} w_\e\left|\f{\pt\vp_\e}{\pt\nu}\right|^2H\cdot \nu\df S\df t\\
		&\hspace{4.5mm}+\f{2-\al}{2}\iint_Q w_\e|\nabla\vp_\e|^2\df x\df t+\f{3\al}{16\e^3}\iint_Qw_\e |\nabla\vp_\e|^2\psi_\e^{-1}\left(\e^2-|x|^2\right)^2\df x\df t.
	\end{split}
\end{equation}

Together with \eqref{04.15.3} and \eqref{04.15.2} and \eqref{04.15.1}, we get
\begin{equation}\label{04.15.4}
	\begin{split}
		&\f{1}{2}\iint_{\partial Q} w_\e\left|\f{\pt\vp_\e}{\pt\nu}\right|^2H\cdot \nu\df S\df t\\
		&=-\f{1}{2}\Image \int_\Om \vp_\e(H\cdot \nabla\ol\vp_\e)\df x\bigg|_{t=0}^{t=T}+\f{2-\al}{2}\iint_Q w_\e|\nabla\vp_\e|^2\df x\df t\\
		&\hspace{4.5mm}+\f{3\al}{16\e^3}\iint_Q w_\e |\nabla\vp_\e|^2\psi_\e^{-1}\left(\e^2-|x|^2\right)^2\df x\df t.
	\end{split}
\end{equation}

Now, multiplying $\pt_t\ol\vp_\e$ on the both sides of \eqref{11.14.2}, integration by parts, from $\pt_t\ol\vp_\e=0$ on $\pt Q$ by Theorem \ref{04.13.T1}, we get
\begin{equation*}
	\begin{split}
		0
		&=\im\iint_Q |\pt_t\vp_\e|^2\df x\df t+\iint_Q w_{\epsilon} \nabla\vp_\e\cdot \nabla \pt_t\ol\vp_\e\df x\df t\\
		&=\im\iint_Q|\pt_t\vp_\e|^2\df x\df t+\f{1}{2}\int_\Om w_{\epsilon} |\nabla\vp_\e|^2\df x\bigg|_{t=0}^{t=T}, 
	\end{split}
\end{equation*}
and then 
\begin{equation}\label{04.15.5}
	\|\vp_\e(t)\|_{H_0^1(\Om;w_\e)}=\|\vp_\e(0)\|_{H_0^1(\Om;w_\e)}, \mbox{ for all } t\in [0,T]. 
\end{equation}

Note that from \eqref{08.31.10}, and $w\leq w_\e$ for all $\e\in (0,R)$ by Lemma \ref{12.17.L1}, we obtain 
\begin{equation*}
	\begin{split}
		\left|\Image \int_\Om \vp_\e(H\cdot \nabla \ol\vp_\e)\df x\right| 
		&\leq \left(\int_\Om |x|^{2-\al}|\vp_\e|^2\df x\right)^\f{1}{2}\left(\int_\Om |x|^\al |\nabla\vp_\e|^2\df x\right)^\f{1}{2}\\
		&\leq \de\|\vp_\e(t)\|_{H_0^1(\Om;w_\e)}^2+\f{C}{\de}\|\vp_\e(t)\|_{L^2(\Om)}^2
	\end{split}
\end{equation*}
for all $t\in [0,T]$ and all $\de>0$, 
where the positive constant $C$ depends only on $\al$ and $M$. Together with \eqref{04.13.2} and  \eqref{04.15.4} and \eqref{04.15.5} and $w_\e=w$ on $\pt Q$ for all $\e\in (0,R)$,  we get for all $\de>0$, 
\begin{equation*}
	\begin{split}
		\theta\iint_{\Sigma_0} \left|\f{\pt\vp_\e}{\pt\nu}\right|^2\df S\df t
		&\geq \f{1}{2}\iint_{\partial Q} w_\e\left|\f{\pt\vp_\e}{\pt\nu}\right|^2H\cdot \nu\df S\df t\\
		&\geq \left(\f{2-\al}{2}T-\de\right)\|\vp_\e^0\|_{H_0^1(\Om;w_\e)}^2-\f{C}{\de}\|\vp_\e^0\|_{L^2(\Om)}^2\\
		&\geq \left(\f{2-\al}{2}T-\de\right)\|\vp_\e^0\|_{H_0^1(\Om;w)}^2-\f{C}{\de}\|\vp_\e^0\|_{L^2(\Om)}^2.
	\end{split}
\end{equation*}
i.e., 
\begin{equation}\label{04.15.6}
	\left(\f{2-\al}{2}T-\de\right)\|\vp_\e^0\|_{H_0^1(\Om;w)}^2\leq \theta\iint_{\Sigma_0}\left|\f{\pt\vp_\e}{\pt\nu}\right|^2\df S\df t+\f{C}{\de}\|\vp_\e^0\|_{L^2(\Om)}^2
\end{equation}
by $w\leq w_\e$ for all $\e\in (0,R)$ and all $\de>0$, 
where the positive constant $C$ depends only on $\al$ and $M$. 

\subsection{Approximation}
\begin{theorem}\label{04.15.T1}	
Let $\vp^0\in H_0^1(\Om;w)$ and assume that $3\alpha+N>4$. Then there exists a positive constant $C$, depending only on $\alpha$, $N$, and $M$, such that
\begin{equation}\label{04.15.9}
	\left(\f{2-\al}{2}T-\de\right)\|\vp^0\|_{H_0^1(\Om;w)}^2\leq \theta\iint_{\Sigma_0}\left|\f{\pt\vp}{\pt\nu}\right|^2\df S\df t+\f{C}{\de}\|\vp^0\|_{L^2(\Om)}^2
\end{equation}
for all $\de>0$, where $\vp$ is the weak solution of \eqref{12.14.1} with respect to $(\vp^0,f=0)$, and $\Sigma_0, \theta$ are defined in \eqref{04.15.8} and \eqref{04.15.7}. 
\end{theorem}

\begin{proof}
Let $\vp^0\in C_0^\infty(\Omega)$. By Lemma \ref{04.14.A1}, $\vp^0\in D(\mathcal A_\epsilon^{\frac 32})$ for every $\epsilon\in(0,R)$. Hence, \eqref{04.15.6} holds for all $\delta>0$. Since $3\alpha+N>4$, the same lemma also gives $\vp^0\in D(\mathcal A^{\frac 32})$. Applying the strong boundary convergence \eqref{04.14.3} in Theorem \ref{04.14.T1} and passing to the limit as $\epsilon_k\downarrow0$, we obtain \eqref{04.15.9} for $\vp^0\in C_0^\infty(\Omega)$.

Finally, let $\vp^0\in H_0^1(\Om;w)$. Then there exists a sequence $\{\vp_n^0\}_{n\in\N^*}\s C_0^\iy(\Om)$ such that $\vp_n^0\ra \vp^0$ strongly in $H_0^1(\Om;w)$. Together with this and \eqref{04.07.3}, we get
\begin{equation}\label{04.15.10}
\|\vp_n-\vp\|_{L^2(0, T; H_0^1(\Om;w))}+\left\|\f{\pt\vp_n}{\pt\nu}-\f{\pt\vp}{\pt\nu}\right\|_{L^2(\pt Q)}\leq C\|\vp_n^0-\vp^0\|_{H_0^1(\Om;w)}\ra 0, \mbox{ as } n\ra\iy, 
\end{equation}
where $\vp_n$ is the weak solution of the equation \eqref{12.14.1} with respect to $(\vp_n^0,f=0)$. Note that for each $n\in\N^*$, from \eqref{04.15.9}, we have 
\begin{equation*}
	\left(\f{2-\al}{2}T-\de\right)\|\vp_n^0\|_{H_0^1(\Om;w)}^2\leq \theta \iint_{\Sigma_0}\left|\f{\pt\vp_n}{\pt\nu}\right|^2\df S\df t+\f{C}{\de}\|\vp_n^0\|_{L^2(\Om)}^2.
\end{equation*}
Letting $ n \rightarrow +\infty$, from \eqref{04.15.10}, we get \eqref{04.15.9} for $\vp^0\in H_0^1(\Om;w)$. 
\end{proof}

\begin{lemma}\label{04.15.L2}
Under the conditions in Theorem \ref{04.15.T1}, we have
\begin{equation*}	
	\|\vp(t)\|_{H_0^1(\Om;w)}=\|\vp^0\|_{H_0^1(\Om;w)}
\end{equation*}
for all $t\in [0,T]$. 
\end{lemma}

\begin{proof}
Let $\vp^0\in C_0^\infty(\Omega)$. Since $3\alpha+N>4$, Lemma \ref{04.14.A1} gives $\vp^0\in D(\mathcal A^{\frac 32})$. The weak solution of the equation \eqref{12.14.1} has the form
\begin{equation*}
	\vp(t)=\sum_{n=1}^\iy e^{\im \la_n t}\vp_n^0\Phi_n, \mbox{ with } \vp_n^0=(\vp^0,\Phi_n)_{L^2(\Om)} \mbox{ for all } n\in\N^*, 
\end{equation*}
then from Lemma \ref{06.28.L1} we obtain 
\begin{equation}\label{04.15.13}
	\|\vp(t)\|_{H_0^1(\Om;w)}^2=\sum_{n=1}^\iy \la_n |\vp_n^0|^2=\|\vp^0\|_{H_0^1(\Om;w)}^2, \mbox{ for all } t\in [0,T]. 
\end{equation}
	
Now, let $\vp^0\in H_0^1(\Om;w)$. Then there exists a sequence $\{\wt\vp_n^0\}_{n\in\N^*}\s C_0^\iy(\Om;w)$  such that $\wt\vp_n^0\ra \vp^0$ strongly in $H_0^1(\Om;w)$. From \eqref{04.07.3} in Theorem \ref{04.07.T1}, we get 
\begin{equation}\label{04.15.14}
	\esssup_{t\in [0,T]}\|\wt\vp_n(t)-\vp(t)\|_{H_0^1(\Om;w)}\leq C\|\wt \vp_n^0-\vp^0\|_{H_0^1(\Om;w)}.
\end{equation}  
Note that for each $n\in\N^*$, we have 
\begin{equation*}
	\|\wt\vp_n(t)\|_{H_0^1(\Om;w)}=\|\wt\vp_n^0\|_{H_0^1(\Om;w)}, \mbox{ for all } t\in [0,T], 
\end{equation*}
together with this and \eqref{04.15.14}, we get
\begin{equation*}
	\|\vp(t)\|_{H_0^1(\Om;w)}=\lim_{n\ra\iy}\|\wt\vp_n(t)\|_{H_0^1(\Om;w)}=\lim_{n\ra\iy}\|\wt\vp_n^0\|_{H_0^1(\Om;w)}=\|\vp^0\|_{H_0^1(\Om;w)}.
\end{equation*}
This completes the proof of this lemma. 
\end{proof}

\begin{theorem}\label{04.15.T2}
Under the conditions in Theorem \ref{04.15.T1}, there exists a positive constant $C$, depends only on $\al, N$ and $M$, such that 
\begin{equation}\label{04.15.15}
	\begin{split}
			\left(\f{2-\al}{2}T-C\right)\|\vp^0\|_{H_0^1(\Om;w)}^2\leq \theta \iint_{\Sigma_0}\left|\f{\pt\vp}{\pt\nu}\right|^2\df S\df t. 
	\end{split}
\end{equation}
\end{theorem}

\begin{proof}
From $\|\vp^0\|_{L^2(\Om)}\leq C\|\vp^0\|_{H_0^1(\Om;w)}$ by \eqref{08.15.10} and \eqref{04.15.9}, we get
\begin{equation*}
		\left(\f{2-\al}{2}T-\de\right)\|\vp^0\|_{H_0^1(\Om;w)}^2\leq \theta \iint_{\Sigma_0}\left|\f{\pt\vp}{\pt\nu}\right|^2\df S\df t+\f{C}{\de}\|\vp^0\|_{H_0^1(\Om;w)}^2, 
\end{equation*}
	i.e., 
	\begin{equation*}
		\left(\f{2-\al}{2}T-\de-\f{C}{\de}\right)\|\vp^0\|_{H_0^1(\Om;w)}^2\leq \theta \iint_{\Sigma_0}\left|\f{\pt\vp}{\pt\nu}\right|^2\df S\df t, 
	\end{equation*}
where the positive constant $C$ depends only on $\al, N$ and $M$. This shows  \eqref{04.15.15} by taking $\de=\sqrt{C}$. 
\end{proof}

\begin{remark}
Let
	\[
	T_0=\frac{2C}{2-\alpha}.
	\]
If $T>T_0$, then
	\[
	\frac{(2-\alpha)T}{2}-C>0.
	\]
	Therefore, the estimate in Theorem \ref{04.15.T2} yields the observability inequality
\begin{equation*}
\|\vp^0\|_{H_0^1(\Om;w)}^2\leq 	C_T \iint_{\Sigma_0}\left|\f{\pt\vp}{\pt\nu}\right|^2\df S\df t, 
\end{equation*}
where
	\[
	C_T=
	\frac{\theta}{\frac{(2-\alpha)T}{2}-C}>0.
	\]
	Hence, for every $T>T_0$, the system satisfies the observability inequality.
\end{remark}

\section{Conclusion}\label{S5}
We have established a non-standard boundary observability inequality for the multi-dimensional degenerate Schr\"odinger equation with an interior degeneracy point. The proof comprises three logical stages. We first develop the functional-analytic setting and prove well-posedness. We then construct a family of non-degenerate approximating problems and show convergence of their solutions. Finally, we derive an observability inequality for the limit problem that quantifies the influence of the degeneracy strength on the observation time. Together, these steps overcome the loss of ellipticity caused by the vanishing principal symbol.

Future work can investigate the sharpness of the observation time threshold to determine the precise minimal time required for observability. From a computational perspective, our constructive approximation scheme provides a basis for numerical discretization to analyze the conditioning of boundary reconstruction near the singularity.

\section*{Acknowledgements}
The authors have disclosed no relevant relationships.

\section*{Data availability statement}
The data that support the findings of this study are available.

\end{document}